\documentclass[DIV=19,12pt]{scrartcl}

\usepackage[T1]{fontenc}
\usepackage[utf8]{inputenc}
\usepackage[english]{babel}
\usepackage{times,amsmath,amsfonts,amssymb,amsthm,txfonts,lmodern,graphicx,tikz,algorithm,algpseudocode,mathtools}
\usepackage[inline]{enumitem}
\usepackage[hidelinks]{hyperref}
\usetikzlibrary{calc}

\usepackage{pgfplots}
\pgfplotsset{compat=newest}
\usetikzlibrary{plotmarks}
\usetikzlibrary{arrows.meta}
\usepgfplotslibrary{patchplots}
\usepackage{grffile}

\title{Adaptive Refinement for Unstructured T-Splines with Linear Complexity}
\author{Roland Maier%
\footnote{Institute of Mathematics, Friedrich Schiller University Jena, Ernst-Abbe-Platz 2, 07743 Jena, Germany}
, Philipp Morgenstern%
\footnote{Institut für Angewandte Mathematik, Leibniz Universität Hannover, Welfengarten 1, 30167 Hannover, Germany}
, Thomas Takacs\footnote{Johann Radon Institute for Computational and Applied Mathematics, \"Osterreichische Akademie der Wissenschaften, Altenberger Stra{\ss}e 69, 4040 Linz, Austria}}

\newcommand          \elements     {\mathcal Q}
\newcommand          \edges        {\mathcal E}
\newcommand          \nodes        {\mathcal N}
\newcommand          \anchors      {\nodes_{\mathrm A}}
\newcommand          \enodes       {\nodes_{{\mathrm{EO}}}}
\newcommand          \interioredges{\edges_{\mathrm{int}}}
\newcommand          \interiornodes{\nodes_{\mathrm{int}}}
\newcommand          \boundaryedges{\edges_{\mathrm{bd}}}
\newcommand          \boundarynodes{\nodes_{\mathrm{bd}}}
\newcommand          \element      {\textsl{\texttt{Q}}}
\newcommand          \elementSmall {\textsl{\texttt{q}}}
\newcommand          \edge         {\textsl{\texttt{E}}}
\newcommand          \node         {\textsl{\texttt{N}}}
\newcommand          \edgeSmall    {\textsl{\texttt{e}}}
\newcommand          \nodeSmall    {\textsl{\texttt{n}}}
\newcommand          \nodeinRd     {\mathbf{n}}
\newcommand          \valence[1]   {v_{#1}}
\newcommand          \length[1]    {\mathrm{L}(#1)}
\newcommand          \domain[1]    {\mathcal{D}(#1)}
\newcommand          \pullback[1]  {\widehat{#1}}
\newcommand          \mapping[1]   {\mathbf{F}_{#1}}
\newcommand          \level[1]     {\ell(#1)}
\newcommand          \kk           {\mathrm{k}}
\newcommand          \kspecial     {\kappa_0}
\newcommand          \bmesh        {\textsc{Bezier}(\mesh)}
\newcommand          \belements    {\textsc{Bezier}(\elements)}
\newcommand          \bedges       {\textsc{Bezier}(\edges)}
\newcommand          \bnodes       {\textsc{Bezier}(\nodes)}
\newcommand          \splines      {\mathcal S}

\newcommand          \ceilfrac [2] {\mathchoice{\bigl\lceil\tfrac{#1}{#2}\bigr\rceil}{\lceil\frac{#1}{#2}\rceil}{\lceil\frac{#1}{#2}\rceil}{\lceil\frac{#1}{#2}\rceil}}
\newcommand          \Cref     [1] {C_{\ref{#1}}}
\newcommand          \cref     [1] {c_{\ref{#1}}}
\newcommand          \cS           {\mathcal{S}}
\newcommand          \cB           {\mathcal{B}}
\DeclareMathOperator \di           {di}
\DeclareMathOperator \dist         {dist}
\DeclareMathOperator \diam         {diam}
\DeclareMathOperator \ext          {ext}
\DeclareMathOperator \ep           {ep}

\newcommand          \integer      {\mathbb Z}
\DeclareMathOperator \KV           {\mathsf K}
\DeclareMathOperator \kv           {\mathsf k}
\newcommand          \mesh         {\mathcal{M}}
\DeclareMathOperator \midp         {mid}
\newcommand          \nat          {\mathbb N}
\DeclareMathOperator \neighb       {neighb}
\newcommand          \Q            {\textsl{\texttt{Q}}}
\newcommand          \q            q 
\newcommand          \real         {\mathbb R}
\DeclareMathOperator \refine       {\textsc{refine}}
\let                 \sei          \coloneqq
\DeclareMathOperator \subdiv       {\textsc{subdiv}}

\newcommand          \bzeta        {\boldsymbol{\zeta}}

\newcommand     \denoteuniform[1]  {\mathsf{u}[#1]}
\newcommand          \meshuni [1]  {\mesh^d_{\denoteuniform{#1}}}
\newcommand     \elementuni [1]    {\elements^d_{\denoteuniform{#1}}}
\newcommand          \edgeuni [1]  {\edges^d_{\denoteuniform{#1}}}
\newcommand          \nodeuni [1]  {\nodes^d_{\denoteuniform{#1}}}
\newcommand          \mfrac [2]    {\mathchoice{\tfrac{#1}{#2}}{\tfrac{#1}{#2}}{#1/#2}{#1/#2}}
\newcommand          \tbigcup      {\mathchoice{{\textstyle\bigcup}}{\bigcup}{\bigcup}{\bigcup}}
\newcommand          \sk           {\mathsf{Sk}}
\newcommand          \dc           {\mathsf{D}}
\newcommand          \emb          {\mathsf{emb}}

\DeclareMathOperator \Exists       \exists
\DeclareMathOperator \Forall       \forall

\theoremstyle{definition}
\newtheorem{df}{Definition}
\numberwithin{df}{section}
\newtheorem{assu}[df]{Assumption}
\newtheorem{rem}[df]{Remark}
\theoremstyle{theorem}
\newtheorem{theorem}[df]{Theorem}
\newtheorem{prop}[df]{Proposition}

\begin{document}

\maketitle

\begin{abstract}
\textbf{\abstractname.}
We present an adaptive refinement algorithm for T-splines on unstructured 2D meshes. While for structured 2D meshes, one can refine elements alternatingly in horizontal and vertical direction, such an approach cannot be generalized directly to unstructured meshes, where no two unique global mesh directions can be assigned. To resolve this issue, we introduce the concept of 
direction indices, i.e., integers associated to each edge, which are inspired by theory on higher-dimensional structured T-splines. 
Together with refinement levels of edges, these indices essentially drive the refinement scheme.
We combine these ideas  with an edge subdivision routine that allows for I-nodes, yielding a very flexible refinement scheme that nicely distributes the T-nodes, preserving global linear independence, 
analysis-suitability (local linear independence) except in the vicinity of extraordinary nodes, 
sparsity of the system matrix, and shape regularity of the mesh elements. Further, we show that the refinement procedure has linear complexity in the sense of guaranteed upper bounds on 
a) the distance between marked and additionally refined elements, and on 
b) the ratio of the numbers of generated and marked mesh elements.
\end{abstract}


\section{Introduction}

T-splines refer to a realization of B-splines on irregular meshes. They were introduced in~\cite{SZBN:2003} in the context of Computer-Aided Design and allow for local mesh refinement~\cite{SCFNZL:2004} without the requirement of a hierarchical basis. They were successfully applied in the context of Isogeometric Analysis~\cite{BCCEHLSS:2010,DJS:2010}, but in the beginning also showed weaknesses such as linear dependencies~\cite{BCS:2010} or even non-nestedness of approximation spaces~\cite{Li:Scott:2011} in certain cases. The problem of linear dependence was overcome by the introduction of analysis-suitability~\cite{ZSHS:2012}, which requires that T-junction extensions do not intersect, and the more abstract but equivalent concept of dual-compatibility~\cite{BBCS:2012}. Then again, the refinement algorithm in~\cite{SLSH:2012} yielded nested spline spaces and also preserved linear independence. Further works on the theoretical background considered arbitrary polynomial degrees~\cite{BBSV:2013} and the construction of T-spline meshes from boundary representations in 3D~\cite{ZWH:2012,WZLH:2013}. At that time, however, the linear independence of higher-dimensional T-splines was only characterized through the dual-compatibility criterion. This was resolved in~\cite{Morgenstern:2016}, where a definition of T-junction extensions and analysis-suitability was presented in three dimensions, which was generalized to arbitrary dimensions in~\cite{Morgenstern:2017}.

Simultaneously, the research on 2D T-splines continued, for instance in the context of analysis-suitable T-spline spaces with globally highest smoothness~\cite{Li:Scott:2014} or locally reduced smoothness~\cite{BBS:2015}. Concerning unstructured meshes, T-splines were originally defined locally on structured, rectangular regions but already in the earliest works extended to more unstructured domains by combining them with subdivision surface constructions near so-called extraordinary nodes. However, to avoid the non-finite representation of subdivision surfaces, special geometrically continuous constructions based on local splits and/or locally higher polynomial degree were introduced in~\cite{SSELBHS:2013,CLBZG:2016,CWTLHKZ:2020}. In~\cite{TosSH17}, unstructured T-splines were constructed that are piecewise bicubic, globally $C^1$-smooth and preserve linear independence also around extraordinary nodes. The construction is based on the concept of having a convenient design space and an analysis space for theoretical purposes. The construction was later extended in~\cite{WLQHZC:2021} to allow more general mesh configurations. While the constructions in~\cite{TosSH17,CWTLHKZ:2020,WLQHZC:2021} allow quite flexible mesh configurations and result in $C^1$-smooth spaces, they are all based on specific splits of the elements near extraordinary nodes and are formulated for cubic T-splines only. In contrast, the construction we present in this paper requires a more restrictive separation of extraordinary nodes, but is stated for T-splines of general polynomial degree.

In this contribution, we consider unstructured spline spaces on two-dimensional unstructured meshes motivated by~\cite{STV:2016}, which yield $C^{p-1}$-continuous splines except for the vicinity of extraordinary nodes, where the continuity is reduced to $C^0$-continuity. These spaces are combined with the theory on higher-dimensional T-splines in~\cite{Morgenstern:2017}. More precisely, we show that unstructured meshes around an extraordinary node can be understood as traces of higher-dimensional structured meshes, for which the refinement routine of~\cite{Morgenstern:2017} is applicable. In practice, this `embedding' of the mesh can be avoided by the idea of so-called direction indices, i.e., integers associated to each edge. These integers can be understood as the equivalent of a certain space dimension in higher-dimensional structured meshes and mark a crucial ingredient in the resulting refinement algorithm together with the refinement levels of edges. 

Throughout this paper, we restrict ourselves to the case of odd-degree spline functions. For even polynomial degrees, we refer the reader to~\cite{BBSV:2013}, where the additional concept of anchor elements is explained. Further note that the construction we present here yields analysis-suitable T-splines. This condition may be weakened by generating only so-called AS++ T-splines, as developed in~\cite{LZ:2018,ZL:2018,LL:2021}.\medskip

This paper is organized as follows. Section~\ref{sec: Preliminaries} introduces the quadrilateral 2D meshes on which we then construct structured and unstructured T-spline spaces in Section~\ref{sec: Manifold splines}. 
In Section~\ref{sec: Refinement}, we present our new refinement algorithm and rigorously analyze properties of the generated meshes and spaces in Section~\ref{sec: mesh properties} before we conclude in Section~\ref{sec: conclusions}. 
Additionally, the connections to manifold splines~\cite{STV:2016} are outlined in the appendix, where also an extension to a larger class of unstructured meshes is presented. Further, some connections to Isogeometric Analysis are drawn.

\section{Preliminaries}
\label{sec: Preliminaries}

The goal of this paper is to define T-splines along with an appropriate adaptive refinement scheme on unstructured T-meshes over polygonal planar domains, planar domains with curved boundaries and surface domains. All domains may have holes. In case of surfaces, we assume that they are orientable. In principle, the concepts that we introduce can be generalized to higher-dimensional objects, such as volumes or space-time domains. We first introduce unstructured T-splines over quadrilateral partitions of the plane. The polygonal partition serves as a template for the underlying topological structure of the mesh. The definition may then be generalized to manifold-like objects.

In the following subsections, we introduce necessary notation and state useful results. For convenience, the main notation is also summarized in Table~\ref{tab:notation}.

\begin{table}[ht]
	\centering
	\begin{tabular}{ll}
		$\element \in \elements$ & mesh elements \\
		$\edge \in \edges$ & mesh edges \\
		$\node = \{ \nodeinRd \} \in \nodes$ & mesh nodes, $\nodeinRd \in \real^d$ \\
		$\mesh = (\elements,\edges,\nodes)$ & mesh (for higher dimensions in principle a longer vector) \\
		$\interioredges$, $\boundaryedges$, $\interiornodes$, $\boundarynodes$ & interior/boundary edges/nodes \\
		{$\anchors$} & {anchors} \\
		{$\enodes$} & {extraordinary nodes}\\
		$\edges(\element)$ & edges of an element $\element$ (contained in $\partial\element$), similarly $\nodes(\element)$, $\nodes(\edge)$ \\
		$\valence{\node}$ & element valence of node $\node$, $\valence{\node}=|\elements(\node)|$ \\
		$\domain{\mesh'}$ & (planar or surface) domain corresponding to a submesh $\mesh'$ \\
		$\domain{\elements'}$ & (planar or surface) domain corresponding to a set of elements $\elements'$\\
		$\pullback{A}$ & parameter domain of a submesh or mesh object (element, set of elements) $A$ \\
		$\length{\edge}$ & length of edge $\edge$ \\
		$\level{\edge}$ & refinement level of edge $\edge$ \\
		$\di(\edge)$ & direction index of edge $\edge$ \\
		{$[i]$ (as subscript)} & {refers to refinement level $i$}
	\end{tabular}
	\caption{Summary of the notation}\label{tab:notation}
\end{table}

\subsection{Partitions and meshes}\label{sec: Partitions-Meshes}

Let $\Omega\subset\real^2$ be an open, bounded domain with a polygonal boundary $\partial \Omega$. On the domain we define a regular partition $(\elements,\edges,\nodes)$ into quadrilaterals $\elements$, edges $\edges$, and nodes $\nodes$. Every node $\node = \{\nodeinRd\} \in \nodes$ corresponds to a point $\nodeinRd\in \real^2$, every edge $\edge \in \edges$ is a line segment between two nodes (excluding the endpoints) and every element $\element \in \elements$ is an open quadrilateral. All edges are either interior edges $\interioredges$ or boundary edges $\boundaryedges$, with $\edges = \interioredges \cup \boundaryedges$, similarly nodes are either interior nodes $\interiornodes$ or boundary nodes $\boundarynodes$, $\nodes = \interiornodes \cup \boundarynodes$. All mesh objects are disjoint, i.e., we have for all $A,A'\in \elements \cup \edges \cup \nodes$, that $A\cap A' = \emptyset$, and we have
\[
\Omega = \bigcup_{A \in \elements \cup \interioredges \cup \interiornodes} A 
\quad \mbox{and}\quad
\partial\Omega = \bigcup_{A \in \boundaryedges \cup \boundarynodes} A.
\]
The elements $\elements$, edges $\edges$, and nodes $\nodes$, together with their connectivity structure, form a topological mesh $\mesh$. The connectivity relations between mesh objects are explained in more detail in the following definition.
\begin{df}[topological structure of the mesh]\label{df: topological-structure-mesh}
	A mesh $\mesh = (\elements,\edges,\nodes)$ is given as a triple of sets of mesh objects (elements, edges, and nodes) together with their topological structure, i.e., the mesh objects satisfy the following connectivity relations:
	\begin{itemize}
		\item An edge or node $A \in \edges \cup \nodes$ is connected to an element $\element\in\elements$, if $A \subset \partial\element$. We denote by $\edges(\element)$ and $\nodes(\element)$ the sets of all edges and nodes, respectively, connected to the element $\element$.
		\item A node $\node \in \nodes$ is connected to an edge $\edge\in\edges$, if $\node \subset \partial\edge$. We denote by $\nodes(\edge)$ the set of all nodes connected to the edge.
		\item We denote by $\elements(\edge)$ the set of those elements $\element \in \elements$, such that $\edge\in\edges(\element)$.
		\item We denote by $\elements(\node)$ and $\edges(\node)$ the sets of those elements and edges, such that $\node\in\nodes(\element)$ and $\node\in\nodes(\edge)$, respectively.
		\item Every node $\node$ has element valence $\valence{\node} \in \nat^+$. That means, $\node$ is connected to $\valence{\node}$ elements and $\valence{\node}$ edges (if it is an interior node) or $\valence{\node}+1$ edges (if it is a boundary node), i.e.,  $|\elements(\node)|=\valence{\node}$, $|\edges(\node)|=\valence{\node}$ for $\node\in\interiornodes$ and $|\edges(\node)|=\valence{\node}+1$ for $\node\in\boundarynodes$.
	\end{itemize}
\end{df}
\begin{prop}\label{prop: mesh-relations}
	We have the following:
	\begin{itemize}
		\item By definition, the connectivity relations are symmetric, e.g., $\node\in\nodes(\element)$ if and only if $\element\in\elements(\node)$.
		\item Every element is connected to at least four edges and nodes, i.e., $|\edges(\element)| = |\nodes(\element)|\geq 4$.
		\item Every edge is connected to exactly two nodes, i.e., $|\nodes(\edge)|=2$.
		\item Every interior edge is connected to two elements, $|\elements(\edge)|=2$ for $\edge\in\interioredges$, and every boundary edge to one element, $|\elements(\edge)|=1$ for $\edge\in\boundaryedges$.
		\item Boundary edges are connected only to boundary nodes and interior nodes are connected only to interior edges.
	\end{itemize}
\end{prop}
\begin{rem}\label{rem: other 2D meshes}
	The concept of a partition of a planar, polygonal domain extends directly to more general meshes over bivariate domains $\Omega \subset \real^d$. The mesh is then composed of bivariate elements, which are curved quadrilateral subdomains without boundary, univariate edges, which are curve segments without endpoints, as well as nodes, which are singletons, each containing a point in $\real^d$. All mesh objects satisfy the relations described in Definition~\ref{df: topological-structure-mesh} and Proposition~\ref{prop: mesh-relations}. Hence, from now on we do not distinguish between planar partitions and more general bivariate meshes.
\end{rem}

\begin{df}[submesh and its domain] 
	Let $\mesh = (\elements,\edges,\nodes)$ be a mesh and let $\elements' \subseteq \elements$ be a subset of elements of the mesh. Then the corresponding \emph{submesh} $\mesh' = (\elements',\edges',\nodes')$ contains all those mesh objects $A \in \edges \cup \nodes$, such that 
	\[
	A \subset \overline{\domain{\mesh'}},
	\]
	where the open set
	\[
	\domain{\mesh'} = \big(\overline{\bigcup_{\element\in \elements'} \element}\big)^\circ
	\]
	is the corresponding \emph{domain} $\domain{\mesh'}\subseteq \Omega$ of the submesh.
	Since the domain depends only on the elements of the submesh, we also use the notation $\domain{\elements'} = \domain{\mesh'}$. 
	The edges and nodes of the submesh are again split into interior ($A \subset \domain{\mesh'}$) and boundary ($A \subset \partial\domain{\mesh'}$) edges and nodes, respectively.
\end{df}

\begin{df}[regular mesh]
	We call a mesh \emph{regular} if the closure of any element $\element\in \elements$ contains exactly four edges and four nodes and the closures of any two elements $\element,\element' \in \elements$
	\begin{enumerate*}[label=\alph*), font=\sffamily]
		\item are disjoint or 
		\item equal or 
		\item share exactly one edge and two nodes or 
		\item share exactly one node.
	\end{enumerate*}
\end{df}
Thus, a regular mesh (or a regular submesh) is a mesh without hanging nodes in its interior.

\begin{df}[extraordinary node]
	A node of a regular mesh is called \emph{extraordinary} if 
	\begin{enumerate*}[label=\alph*), font=\sffamily]
		\item it is a boundary node neighboring more than $2$ quadrilaterals or
		\item it is an interior node which is not neighboring exactly $4$ mesh elements.
	\end{enumerate*}
\end{df}

\subsection{Parameter domains of elements and structured submeshes}

From now on, we consider to have an initial mesh $\mesh_{[0]} = (\elements_{[0]},\edges_{[0]},\nodes_{[0]})$, which is regular. In the following, we introduce parameter domains to mesh elements of the initial mesh. These parameter domains can then be extended from single elements to larger structured submeshes and to refinements of the initial mesh. 
To this end, we assign a length to each edge.
\begin{df}[edge length]
	On the initial mesh $\mesh_{[0]}$ we assign the \emph{length} $\length{\edge} = 1$ to each edge $\edge \in \edges_{[0]}$.
\end{df}
It is also possible to assign different lengths to different edges, with the restriction that the lengths must be consistent, i.e., for each element the same length is assigned to opposing edges. The choice $\length{\edge}=1$ above is for simplicity.
\begin{df}[parameter domain]
	On the initial mesh $\mesh_{[0]}$ we assign the \emph{parameter domain} $\pullback{\element}=\left]0,1\right[^2$ to each element $\element \in \elements_{[0]}$.
\end{df}
The parameter domain is defined such that it is consistent with the edge lengths. If edge lengths different from $\length{\edge} = 1$ are assigned, then the parameter domains are changed accordingly.
\begin{assu}\label{assu: local-mappings}
	We assume for all $\element\in\elements_{[0]}$ that there exists a regular, positively-oriented $C^\infty$-mapping $\mapping{\element} : \pullback{\element} \rightarrow \element$, i.e., there exists a constant $C>0$ such that $\det\nabla \mapping{\element} (\mathbf x) \geq C > 0$ for all $\mathbf x\in \pullback{\element}$. 
	This implies that $\element$ is not degenerate, e.g., no edge of $\element$ is empty.
	Moreover, we assume that for any pair of elements $\element$, $\element'$ sharing a common edge $\edge$, the mappings $\mapping{\element}$ and $\mapping{\element'}$ are continuous along $\edge$, i.e., there exists a common mapping $\mapping{\element,\element'}:\left(0,1\right)\times\left(-1,1\right) \rightarrow \element \cup \edge \cup \element'$ and Euclidean motions $\mathbf{R}_{\element}$ and $\mathbf{R}_{\element'}$, such that 
	\[
	\mapping{\element,\element'} |_{\left(0,1\right) \times \left(0,1\right)} = \mapping{\element}\circ \mathbf{R}_{\element}
	\]
	and 
	\[
	\mapping{\element,\element'} |_{\left(0,1\right) \times \left(-1,0\right)} = \mapping{\element'} \circ \mathbf{R}_{\element'}.
	\]
\end{assu}
By assumption, the Euclidean motions $\mathbf{R}_{\element}$, $\mathbf{R}_{\element'}$ must be combinations of a rotation by an angle which is an integer multiple of $\frac{\pi}{2}$ and a translation.

Now that we are given parameter domains $\pullback{\element}$ and mappings $\mapping{\element}$ for each element $\element$, we can define (mapped) polynomial functions on each element. Moreover, the assumption that there exists a joint mapping for each edge allows us to define continuity of functions across edges. Consequently, we can define \emph{splines} over the mesh $\mesh_0$. Before we do that, we introduce a generalization of the mapping introduced in Assumption~\ref{assu: local-mappings} and extend all definitions to T-meshes, which are obtained from an initial, regular mesh through refinement.

\begin{df}[structured mesh]\label{df: structured-mesh}
	Let $\mesh$ be any submesh of the initial mesh $\mesh_{[0]}$, which has a simply connected domain $\domain{\mesh}$ and does not contain any extraordinary nodes. We call the mesh \emph{structured}, if in addition there exists a continuous mapping
	\[
	\mapping{\elements}: \domain{\pullback{\mesh}} \rightarrow \domain{\mesh},
	\]
	such that $\pullback{\mesh}$ is a mesh derived from a set of axis-aligned boxes $\pullback{\elements}$, with
	\[
	\pullback{\elements} = \{\mathbf{R}_{\element} (\pullback{\element}),\element\in \elements\},
	\]
	where $\mathbf{R}_{\element}$ is a Euclidean motion, $\domain{\pullback{\mesh}}$ is simply connected and $\mapping{\elements}$ satisfies
	\[
	\mapping{\elements} \circ \mathbf{R}_{\element} = \mapping{\element},
	\]
	for all $\element\in\elements$. Here $\mapping{\element}$ is the mapping from Assumption~\ref{assu: local-mappings}.
\end{df}
The definition of a structured mesh carries over directly to any regular mesh, not just submeshes of the initial mesh.
\begin{prop}\label{prop: covering-with-structured-submeshes}
	Let $\mesh=(\elements,\edges,\nodes)$ be a regular mesh. Each element $\element \in \elements$ and each interior edge $\edge \in \interioredges$ is contained inside a structured submesh. Moreover, for each node $\node \in \nodes$, which is not an extraordinary node, the neighborhood $\node \cup \edges(\node) \cup \elements(\node)$ is contained inside a structured submesh.
\end{prop}
A visualization of Proposition~\ref{prop: covering-with-structured-submeshes} is shown in Figure~\ref{fig: ring mesh}. For practical purposes, it is recommended to find a small number of structured submeshes to cover the domain.

\begin{figure}[t]
	\centering
	\includegraphics{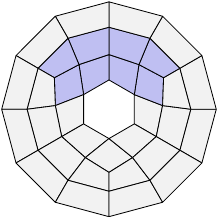}\quad
	\includegraphics{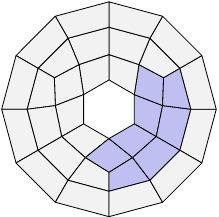}\quad
	\includegraphics{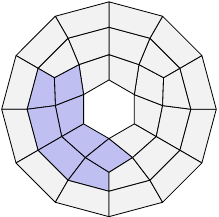}\\[1ex]
	\includegraphics{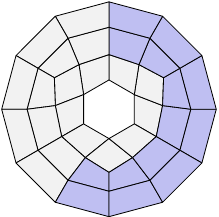}\quad
	\includegraphics{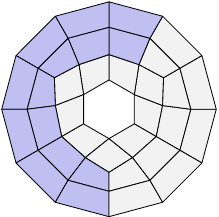}\quad
	\includegraphics{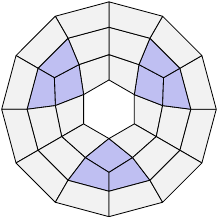}
	\caption{A regular mesh on a ring-shaped domain, which is covered by five structured submeshes (top line, bottom left and center) and three submeshes which are not structured (bottom right).}
	\label{fig: ring mesh}
\end{figure}

\subsection{T-meshes}\label{subsec: T-meshes}

T-meshes are derived from an initial, regular mesh $\mesh_0$ by refinement. In the following, we characterize valid refinements on a mesh.
\begin{df}[T-mesh and splitting operations]\label{df: T-mesh-splitting}
	A T-mesh is either a regular mesh or it is obtained by a splitting operation from a T-mesh, where we allow the following splitting operations:
	\begin{itemize}
		\item An edge $\edge\in \edges$ can be split by bisection, creating two new edges $\edge_1$ and $\edge_2$ of length $\length{\edge_i} = \frac{1}{2}\length{\edge}$ and one new node $\node=\overline{\edge_1}\cap\overline{\edge_2}$.
		\item An element $\element \in \elements$ can be split if the two opposite edges are bisected, creating two new elements $\element_1,\element_2$ and one new edge $\edge=\overline{\element_1}\cap\overline{\element_2}$. The length of the new edge is assigned such that it is consistent with the lengths of the existing edges of the element $\element$. Parameter domains are assigned to the new elements by bisecting the parameter domain $\pullback{\element}$ of the element $\element$. Mappings are then assigned trivially to the new elements by restricting $\mapping{\element}$ to the new parameter domains. 
	\end{itemize}
\end{df}
Thus, all parameter domains of elements of a T-mesh are axis aligned boxes with edge lengths in $\{ 1/2^i, i \in \nat_0 \}$. Since the parameter domains are inherited from the underlying regular mesh, many properties of the regular mesh can be generalized to T-meshes.
\begin{df}[nodes of a T-mesh]\label{df: T-node and I-node}
	A node of a T-mesh is called an \emph{extraordinary node} if and only if it is an extraordinary node of the underlying regular mesh $\mesh_0$. An inner node $\node$ of a T-mesh is called an \emph{I-node} if $|\elements(\node)| = 2$, it is called a \emph{T-node} if it is not an extraordinary node and $|\elements(\node)| = 3$, it is called a \emph{regular node} if $|\elements(\node)| = 4$. Analogously, a boundary node is called an I-node if $|\elements(\node)| = 1$ and a regular node if $|\elements(\node)| = 2$. {The set of extraordinary nodes is denoted $\enodes$.}
\end{df}
We can now introduce the concept of structured T-meshes and structured submeshes of the global T-mesh. 
\begin{df}[structured T-mesh]\label{df: structured T-mesh}
	A submesh of a T-mesh is \emph{structured} if all conditions of Definition~\ref{df: structured-mesh} are satisfied, where for each split element $\element$ the parameter domain $\pullback{\element}$ and mapping $\mapping{\element}$ are given as in Definition~\ref{df: T-mesh-splitting}. Otherwise, the T-mesh is called \emph{unstructured}.
\end{df}
By definition, a structured submesh of an unstructured T-mesh contains no extraordinary nodes in its interior. For each structured submesh $\mesh^*$ we can define standard (structured) T-splines on $\pullback{\mesh^*}$, by defining them on $\domain{\pullback{\mesh^*}}$ and mapping them onto $\domain{\mesh^*}$ using the mapping $\mapping{\elements^*}$. 

We moreover propose the following for T-meshes which are obtained by refining a regular, structured mesh.
\begin{prop}
	Let $\mesh^*$ be a refinement of a mesh~$\mesh'$, following Definition~\ref{df: T-mesh-splitting}. Then~$\mesh^*$ is structured if and only if~$\mesh'$ is structured, with~$\domain{\mesh^*}=\domain{\mesh'}$ and~$\mapping{\elements^*} = \mapping{\elements'}$.
\end{prop}
From this it follows that Proposition~\ref{prop: covering-with-structured-submeshes} extends directly to T-meshes. 
Since the refinement of a T-mesh is based on bisections of edges, every edge in a T-mesh has length $|\edge|=2^{-\level{\edge}}$, where $\level{\edge}$ is the refinement level of the edge $\edge$, i.e., the number of refinement steps that have been performed on it. The refinement level of an edge is defined properly in Definition~\ref{df: refinement-level} below.

\subsection{The mesh metric and separation of extraordinary vertices}

\begin{df}[mesh metric]\label{df: mesh metric}%
	We define a metric on $\Omega$ by the following steps:
	\begin{enumerate}
		\item Given an initial, regular mesh $\mesh_{[0]} = (\elements_{[0]},\edges_{[0]},\nodes_{[0]})$, we understand the \emph{distance} $\dist_{[0]}(A,B)$ between two distinct nodes $A,B \in\nodes_{[0]}$ as the minimal number $n\in\nat$ such that there is a sequence of $n$ vertex-connected elements of $\elements_{[0]}$ with $A$ being a vertex of the first and $B$ a vertex of the last element. We consider two elements vertex-connected if they share a common vertex. For $A=B$, we set $\dist_{[0]}(A,B)=0$.
		This defines a metric on $\nodes_0$.
		\item For the uniform dyadic refinement $\mesh_{[j+1]}$ of $\mesh_{[j]}$, 
		we observe for all nodes $A,B\in\nodes_{[j]}$ that
		\begin{equation}\label{eq: dist j = 1/2 dist j+1}
			\dist_{[j]}(A,B)= \tfrac12\dist_{[{j+1}]}(A,B)~.
		\end{equation}
		We extend the definition of $\dist_{[j]}$ by requiring \eqref{eq: dist j = 1/2 dist j+1} for all nodes $A,B\in\nodes_{[j+1]}$.
		\item The recursive application of step 2 yields for any $j\in\nat$ and $A,B\in\nodes_{[j]}$ that 
		\[\dist_{[0]}(A,B)= 2^{-j}\dist_{[j]}(A,B)~,\]
		which defines a metric on $\bigcup_{j\in\mathbb N_0}\nodes_{[j]}$, this is, on all nodes of all successive dyadic refinements of the initial mesh $\mesh_{[0]}$. 
		This set of nodes is dense in $\Omega$. The continuous extension defines a metric $\dist$ on $\Omega$, which coincides with the piecewise bilinear interpolation of $\dist_{[0]}$. We denote this metric by $\dist(\mathbf{a},\mathbf{b})$ for $\mathbf{a},\mathbf{b}\in\Omega$.
	\end{enumerate}
	On a uniform structured grid of unit squares, this metric corresponds to the $\infty$-norm. However, $\dist$ is in general neither homogeneous nor invariant under translation and hence does not correspond to a norm, see Figure~\ref{fig: mesh metric} for an illustration.
\end{df}
\begin{figure}[t]
	\centering
	\raisebox{.0125\textwidth}{\includegraphics[width=.3\textwidth]{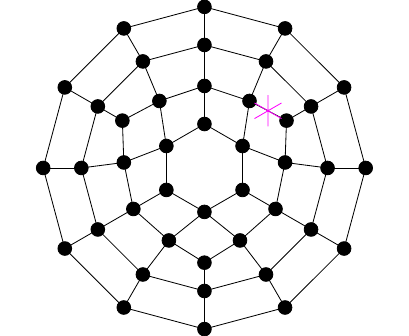}}
	\includegraphics[width=.35\textwidth]{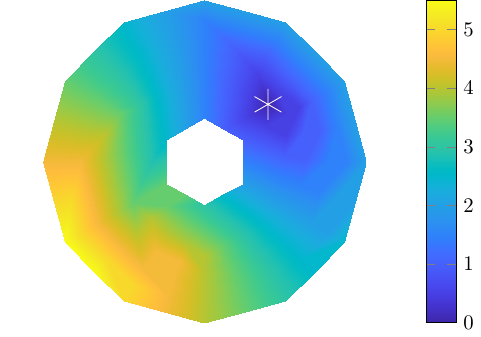}
	\caption{Example for the distance to an indicated point with respect to $\dist$.}
	\label{fig: mesh metric}
\end{figure}

\begin{df}[$k$-disk]
	For $k\in\nat$ and any node $\node\in\nodes$ of a mesh $\mesh$, the \emph{$k$-disk} $\dc_k(\node)$ \emph{around $\node$} is defined recursively, where we have $\dc_1(\node) = \elements(\node)$ and 
	define
	\[
	\nodes(\dc_{k-1}(\node)) = \bigcup_{\element' \in \dc_{k-1}(\node)} \nodes(\element').
	\]
	and
	\[
	\dc_k(\node) = \bigcup_{\node'\in\nodes(\dc_{k-1}(\node))}\elements(\node'),
	\]
	for $k \geq 2$.
\end{df}
Thus, on the initial, regular mesh $\mesh_0$, the $k$-disk around $\node$ is the set of all elements within a distance to $\node$ smaller or equal to $k$. 

\begin{figure}[t]
	\centering
	\begin{tikzpicture}[baseline=0]
		\draw (0:0) +(90:2) coordinate (t) +(162:2) coordinate (ol) +(18:2) coordinate (or) +(-54:2) coordinate (ur) +(-126:2) coordinate (ul)
		(t)--(ol)--(ul)--(ur)--(or)--(t)
		($.5*(t)$) -- ($.5*(ol)$) -- ($.5*(ul)$) -- ($.5*(ur)$) -- ($.5*(or)$) node[circle, fill=blue, inner sep=2pt]{} -- ($.5*(t)$)
		(0,0) -- ($.5*(t) + .5*(ol)$) (0,0) -- ($.5*(ul) + .5*(ol)$) (0,0) -- ($.5*(ul) + .5*(ur)$) (0,0) -- ($.5*(or) + .5*(ur)$) (0,0) -- ($.5*(t) + .5*(or)$)
		($.5*(t)$) -- ($.75*(t) + .25*(ol)$)  ($.5*(t)$) -- ($.75*(t) + .25*(or)$)
		($.5*(ol)$) -- ($.25*(t) + .75*(ol)$)  ($.5*(ol)$) -- ($.75*(ol) + .25*(ul)$)
		($.5*(ul)$) -- ($.75*(ul) + .25*(ol)$)  ($.5*(ul)$) -- ($.75*(ul) + .25*(ur)$)
		($.5*(ur)$) -- ($.75*(ur) + .25*(ul)$)  ($.5*(ur)$) -- ($.75*(ur) + .25*(or)$)
		($.5*(or)$) -- ($.75*(or) + .25*(ur)$)  ($.5*(or)$) -- ($.75*(or) + .25*(t)$);
	\end{tikzpicture}\quad
	\begin{tikzpicture}[baseline=0]
		\tikzset{every node/.style={inner sep=1pt, fill=white}}
		\draw (0:0) +(90:2) coordinate (t) node {3}  +(162:2) coordinate (ol) node {3} +(18:2) coordinate (or) node {1} +(-54:2) coordinate (ur) node {3} +(-126:2) coordinate (ul) node {3} 
		(t)--(ol)--(ul)--(ur)--(or)--(t)
		($.5*(t)$) node {2} -- ($.5*(ol)$) node {2} -- ($.5*(ul)$) node {2} -- ($.5*(ur)$) node {2} -- ($.5*(or)$)  node {0}  -- ($.5*(t)$)
		(0,0) node {1}  -- ($.5*(t) + .5*(ol)$) node {3} (0,0) -- ($.5*(ul) + .5*(ol)$) node {3} (0,0) -- ($.5*(ul) + .5*(ur)$) node {3} (0,0) -- ($.5*(or) + .5*(ur)$) node {1} (0,0) -- ($.5*(t) + .5*(or)$) node {1}
		($.25*(t) + .25*(ol)$) node {2} ($.25*(ul) + .25*(ol)$) node {2} ($.25*(ul) + .25*(ur)$) node {2} ($.25*(or) + .25*(ur)$) node {1} ($.25*(t) + .25*(or)$) node {1} 
		($.5*(t)$) -- ($.75*(t) + .25*(ol)$) node {3}  ($.5*(t)$) -- ($.75*(t) + .25*(or)$) node {2}
		($.5*(ol)$) -- ($.25*(t) + .75*(ol)$) node {3}  ($.5*(ol)$) -- ($.75*(ol) + .25*(ul)$) node {3}
		($.5*(ul)$) -- ($.75*(ul) + .25*(ol)$) node {3}  ($.5*(ul)$) -- ($.75*(ul) + .25*(ur)$) node {3}
		($.5*(ur)$) -- ($.75*(ur) + .25*(ul)$) node {3}  ($.5*(ur)$) -- ($.75*(ur) + .25*(or)$) node {2}
		($.5*(or)$) -- ($.75*(or) + .25*(ur)$) node {1}  ($.5*(or)$) -- ($.75*(or) + .25*(t)$) node {1}   ;
	\end{tikzpicture}\quad
	\begin{tikzpicture}[baseline=0]
		\draw (0:0) +(90:2) coordinate (t) +(162:2) coordinate (ol) +(18:2) coordinate (or) +(-54:2) coordinate (ur) +(-126:2) coordinate (ul);
		\fill[blue!10] (t)--(ol)--(ul)--(ur)--(or)--(t);
		\fill[blue!30] ($.5*(ul)$) -- ($.5*(ur)$) -- ($.75*(ur) + .25*(or)$) -- (or) -- ($.75*(t) + .25*(or)$) -- ($.5*(t)$) -- ($.5*(ol)$) -- ($.5*(ul)$);
		\fill[blue!80] (0,0) -- ($.5*(or) + .5*(ur)$) -- (or) -- ($.5*(t) + .5*(or)$) -- (0,0);
		\draw (t)--(ol)--(ul)--(ur)--(or)--(t)
		($.5*(t)$) -- ($.5*(ol)$) -- ($.5*(ul)$) -- ($.5*(ur)$) -- ($.5*(or)$) -- ($.5*(t)$)
		(0,0) -- ($.5*(t) + .5*(ol)$) (0,0) -- ($.5*(ul) + .5*(ol)$) (0,0) -- ($.5*(ul) + .5*(ur)$) (0,0) -- ($.5*(or) + .5*(ur)$) (0,0) -- ($.5*(t) + .5*(or)$)
		($.5*(t)$) -- ($.75*(t) + .25*(ol)$)  ($.5*(t)$) -- ($.75*(t) + .25*(or)$)
		($.5*(ol)$) -- ($.25*(t) + .75*(ol)$)  ($.5*(ol)$) -- ($.75*(ol) + .25*(ul)$)
		($.5*(ul)$) -- ($.75*(ul) + .25*(ol)$)  ($.5*(ul)$) -- ($.75*(ul) + .25*(ur)$)
		($.5*(ur)$) -- ($.75*(ur) + .25*(ul)$)  ($.5*(ur)$) -- ($.75*(ur) + .25*(or)$)
		($.5*(or)$) -- ($.75*(or) + .25*(ur)$)  ($.5*(or)$) -- ($.75*(or) + .25*(t)$);
	\end{tikzpicture}
	
	\caption{Left: example of a regular mesh with one node $v$ marked.
		Center: distance of each node to $v$.   Right: 1-disk (blue), 2-disk (blue and light blue), and 3-disk (blue, light blue and very light blue) of $v$.}
\end{figure}
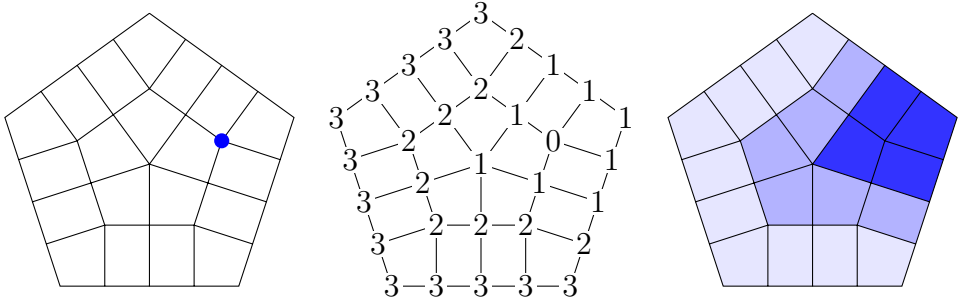

\section{T-spline spaces}
\label{sec: Manifold splines}

We recall below the notion of splines over general T-meshes and related definitions.
Section~\ref{subsec: high-dim struct T-splines} recalls the construction  of T-splines over structured meshes in arbitrary dimensions
and Section~\ref{subsec: unstruct T-splines} introduces a construction of T-splines over unstructured quadrilateral meshes.

\subsection{Spline spaces over T-meshes}\label{subsec: Spline spaces over T-meshes}

We consider a two-dimensional mesh $\mesh=(\elements,\edges,\nodes)$ or a more general mesh as in Remark~\ref{rem: other 2D meshes}. Over such a mesh we want to define splines, that is, piecewise polynomial functions with certain prescribed continuity. Having given parameter domains for all elements and mappings as introduced in Section~\ref{subsec: T-meshes} that relate parameter domains across edges, we can define polynomials on each element and continuity of functions across edges.

Using the mapping $\mapping{\element}$ over the parameter domain~$\pullback{\element}$, we can define polynomials on the mesh element~$\element$. 
\begin{df}[polynomial on an element]
	Let $\element \in \elements$ and let $p\in \nat_0$. A function $\varphi: \element\rightarrow\real$ is called \emph{polynomial} of bi-degree $(p,p)$ if $\varphi \circ \mapping{\element} \in \mathbb{P}^{(p,p)}$. In short, we denote this by $\varphi \in \mathbb{P}^{(p,p)}(\element)$.
\end{df}
\begin{df}[spline on a mesh]
	A function $\varphi:\Omega\rightarrow\real$ is a \emph{continuous, piecewise polynomial spline} of bi-degree $(p,p)$ if $\varphi$ is a polynomial of bi-degree $(p,p)$ on all elements $\element \in \elements$ and $\varphi \in C^0 (\Omega)$.
\end{df}
Continuity of a function on $\Omega$ can be defined directly on the domain itself. However, differentiability (and higher-order continuity) depends on the underlying mesh and embedding. Thus, we define it on the parameter domains of structured submeshes.
\begin{df}[continuity across an edge]\label{df: continuity edge}
	Let $\edge\in\interioredges$ be an interior edge of the mesh and let $\elements(\edge)$ contain the two neighboring elements. A function $\varphi:\domain{\elements(\edge)}\rightarrow\real$ is $C^k$-smooth across the edge $\edge$, if $\varphi \circ \mapping{\elements(\edge)} \in C^k(\domain{\pullback{\elements(\edge)}})$. In short, we denote this by $\varphi \in C^k(\edge)$.
\end{df}
Thus, a function is $C^k$-smooth across an edge of the mesh if its pullback is $C^k$-smooth in the parameter domain. For each integer $k\geq 0$ and for each structured submesh $\mesh^*$, we can analogously define the space $C^k(\domain{\mesh^*})$.

\begin{df}[spline space over a T-mesh]\label{df: spline space T-mesh}
	Let $\mesh=(\elements,\edges,\nodes)$ be a T-mesh over $\Omega$, $p\in \nat_0$ be a prescribed polynomial degree and $\kk:\edges \rightarrow \nat_0$ be a function associating an order of continuity to every edge of the mesh. The spline space $\splines^p(\mesh,\kk)$ is defined as 
	\[
	\splines^p(\mesh,\kk) = \left\{ \varphi : \Omega \rightarrow \real \mid \begin{array}{ll}
		\varphi \in \mathbb{P}^{(p,p)}(\element) &\mbox{ for each } \element\in\elements \mbox{ and } \\
		\varphi \in C^{\kk(\edge)}(\edge) &\mbox{ for each } \edge\in\edges
	\end{array} \right\}.
	\]
\end{df}
Certain constructions for bases on splines over T-meshes can be found in~\cite{STV:2016}. 
Note that the above construction can be further generalized using meshes defined on so-called parameter manifolds. This generalization is briefly covered in Section~\ref{sec: manifold interpretation} in the appendix.

In the following, we introduce T-splines over T-meshes. First, we consider extensions to higher dimensions and, in a second step, unstructured T-splines accounting for the occurrence of extraordinary nodes. The T-mesh serves as a control structure of the T-spline space. 
The resulting T-splines are not piecewise polynomials over the control T-mesh, but over the so-called B\'ezier mesh, which is another T-mesh obtained from the control T-mesh through refinement. This is further explained below.
We restrict ourselves to constructions of odd degree $p$, where, for structured T-meshes, the nodes of the mesh are anchors of the T-splines. Each anchor corresponds to one T-spline basis function. 
For even and mixed degree, see \cite{BBSV:2013,Goermer:Morgenstern:2021}.
\begin{df}[B\'ezier mesh]
	The mesh $\bmesh = (\belements,\bedges,\bnodes)$ is the coarsest partition of the domain $\Omega$ such that all T-splines defined over the control T-mesh $\mesh$ are piecewise polynomials on $\belements$. 
\end{df}

The B\'ezier mesh $\bmesh$ depends on the polynomial degree $p$ of the spline space. The control mesh and B\'ezier mesh differ only near T-nodes and I-nodes. If $\mesh$ is regular, then $\bmesh=\mesh$.
For the class of meshes generated by our refinement scheme, the elements of the B\'ezier mesh can be computed with Algorithm~\ref{alg: bezier mesh}.

\begin{algorithm}
	\caption{B\'ezier mesh}\label{alg: bezier mesh}
	\begin{algorithmic}[1] 
		\Procedure{Bezier}{$\mesh$} 
		\For {$j=1,\dots,\ceilfrac p2$}
		\State {$\elements_\top\gets\{\element\in\elements\mid\#\edges(\element)=\#\nodes(\element)=5\}$}
		\Comment{these are all elements that neighbor T- or I-nodes, see also Remark~\ref{rem: tj-exts far away} for details}
		\ForAll {$\element\in\elements_\top$}
		\State $\edge\gets$ unique edge in $\edges(\element)$ that has two opposite edges
		\State $\mesh\gets\subdiv(\mesh,\edge)$ \Comment {see Algorithm~\ref{alg: subdiv}}
		\EndFor
		\EndFor
		\State \textbf{return} $\mesh$
		\EndProcedure
	\end{algorithmic}
\end{algorithm}

\subsection{Higher-dimensional structured T-spline basis}\label{subsec: high-dim struct T-splines}
In this subsection, we explain the construction of T-splines in meshes that consist of axis-aligned $d$-dimensional boxes.
For details, we refer the reader to \cite[Section 5.3]{Morgenstern:2017}. Note that such meshes are composed of $d$-dimensional elements, $(d-1)$-dimensional faces, \ldots, $1$-dimensional edges, and $0$-dimensional nodes. While the elements, faces, and nodes play important roles for the definition of spline spaces of odd degree, the other mesh elements do not need to be defined rigorously. We emphasize that the definitions in Sections~\ref{sec: Preliminaries} and~\ref{subsec: Spline spaces over T-meshes} can be naturally extended to dimensions $d > 2$ and we include a superscript $d$ in such cases to distinguish higher-dimensional meshes and two-dimensional ones.

\begin{df}[uniform meshes]
	For each level $\ell=k+\tfrac jd$, with $k\in\nat$ and $j\in\{0,\dots,d-1\}$, we define the elements of the $d$-dimensional tensor-product mesh $\meshuni{\ell}$ as
	\begin{multline}
		\elementuni{\ell}\sei\Bigl\{(x_1-2^{-k-1},x_1)\times\dots\times(x_j-2^{-k-1},x_j)\times(x_{j+1}-2^{-k},x_{j+1})\times\dots\times(x_d-2^{-k},x_d)
		\\
		\begin{alignedat}[b]{2}
			\mid 2^{k+1} x_i&\in\{-2^{k+1}p+1,\dots,2^{k+1} p\}&&\text{ for }i=1,\dots,j
			\\
			\text{and}\enspace2^k x_i&\in\{-2^kp+1,\dots,2^k p\}&&\text{ for }i=j+1,\dots,d\Bigr\}
		\end{alignedat}
	\end{multline}
	with $n_1,\dots,n_d\in\nat$. We moreover denote by $\edgeuni{\ell}$ the $(d-1)$-dimensional faces between elements and by $\nodeuni{\ell}$ the vertices of all elements in $\elementuni{\ell}$;
	see also Figure~\ref{fig: uniform meshes} for an illustration.
	Note that we intentionally choose the domain $(-p,p)^d$ here since we explicitly require such a domain in Section~\ref{subsec: unstruct T-splines}.
\end{df}

\begin{figure}[t]
	\centering
	\includegraphics{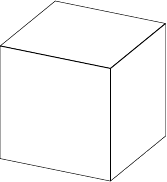}\quad
	\includegraphics{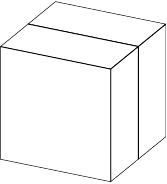}\quad
	\includegraphics{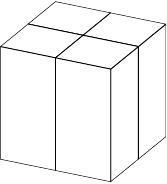}\quad
	\includegraphics{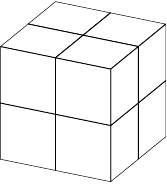}
	\caption{Example for uniform meshes $\meshuni0$, $\meshuni{\mfrac13}$, $\meshuni{\mfrac23}$, and $\meshuni1$, with $\meshuni0$ consisting of a~single element, $d=3$.}
	\label{fig: uniform meshes}
\end{figure}

The class of $d$-dimensional T-meshes $\mesh^d$ considered below is the set of all meshes where $\elements^d$ consists of finitely many elements from uniform meshes as above with possibly different levels such that any two elements of $\elements^d$ are disjoint and the union of all elements of $\elements^d$ and their corresponding edges and vertices is the same domain $(-p,p)^d$ that is covered by each uniform mesh.

\begin{df}[skeleton]
	Given a mesh $\mesh^d$ with elements $\elements^d$, denote the union of all (closed) element faces that are orthogonal to the first dimension by $\sk_1(\mesh^d)\sei\tbigcup_{\Q\in\elements^d}\sk_1(\Q)$, with 
	\begin{equation*}\begin{aligned}[b]
			\sk_1(\Q) &\sei \{x_1,x_1+\tilde x_1\}\times [x_2,x_2+\tilde x_2]\times\dots\times [x_d,x_d+\tilde x_d]\\
			\text{for any } \Q &= (x_1,x_1+\tilde x_1)\times\dots\times (x_d,x_d+\tilde x_d)\in\elements^d.
	\end{aligned}\end{equation*}
	We call $\sk_1(\mesh^d)$ the \emph{1-orthogonal skeleton}. Analogously, we denote the $j$-orthogonal skeleton by $\sk_j(\mesh^d)$ for all $j=1,\dots,d$. 
	Similarly to the above Definition, we abbreviate $\sk_j\sei\sk_j(\mesh^d)$ if only one mesh $\mesh^d$ is considered in the respective context.
	Note that $\sk_1\cap\dots\cap\sk_d=\bigcup\nodes^d$. 
\end{df}

\begin{figure}[t]
	\centering
	\includegraphics[width=.275\textwidth]{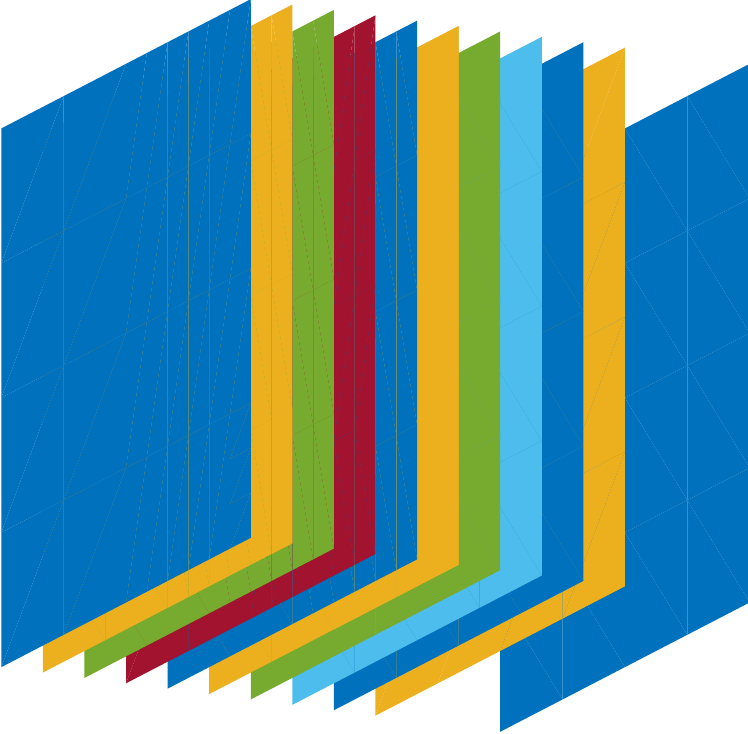}\hspace{.05\textwidth}
	\includegraphics[width=.275\textwidth]{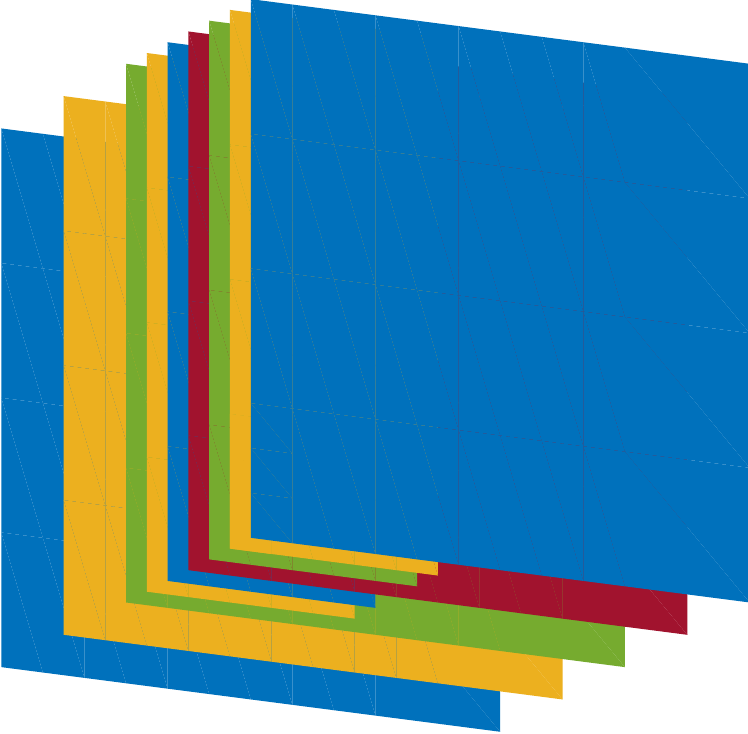}\hspace{.05\textwidth}
	\includegraphics[width=.275\textwidth]{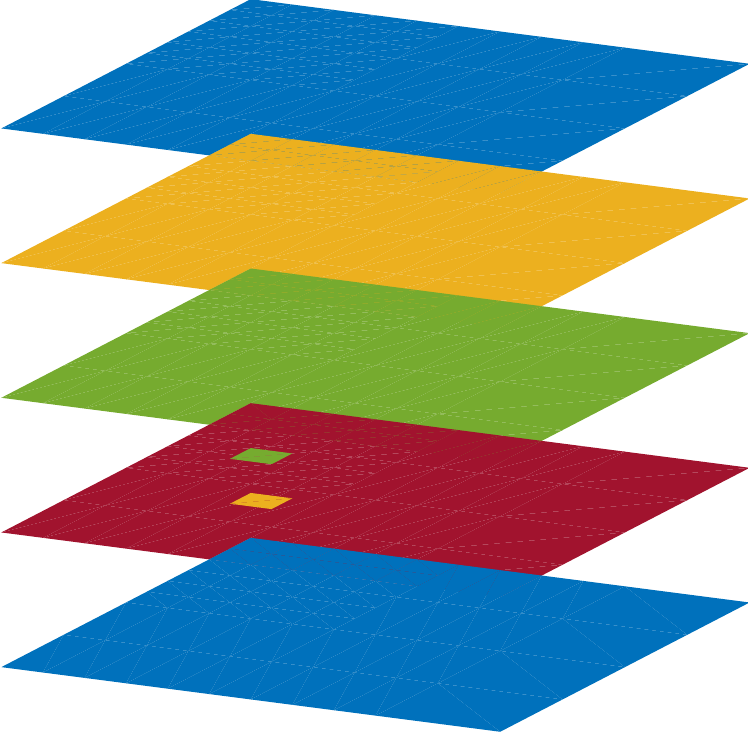}
	\caption{1-orthogonal, 2-orthogonal and 3-orthogonal skeleton of a 3D cube refined in the front corner, taken from \cite{Morgenstern:2016}.}
	\label{tspnd.fig: Xi_xyx_d}
\end{figure}

\begin{df}[global index sets]
	For any $\node=(x_1,\dots,x_d)\in\real^d$ and $j\in\{1,\dots,d\}$, we define
	\begin{multline}
		\KV_j(\node) \sei \bigl\{z\in[-p,p]\mid(x_1,\dots,x_{j-1},z,x_{j+1},\dots,x_d)\in\sk_j\bigr\}
		\\\cup
		\bigl\{-p -\ceilfrac{p}2,\dots,-p-1,p+1,\dots, p+\ceilfrac{p}2\bigr\}.
	\end{multline}
\end{df}

\begin{df}[local index vectors]\label{tspnd.df: local index vector}
	To each node $\node\in\nodes^d$ and each dimension $j=1,\dots,d$, we associate  a local index vector $\kv_j(\node)\in\real^{p+2}$, which consists of the unique $p+2$ consecutive elements in $\KV_j(\node)$ having {the $j$-th component of $\node$} as their $\mfrac{p+3}2$-th (this is, the middle) entry.
\end{df}
\begin{df}[B-spline]
	Given a local index vector $\kv_j(\node)=(k_1,\dots,k_{p+2})$, we denote by $B_{\kv_j(\node)}$ the standard 1D B-spline of degree $p$ that corresponds to the knots $k_1,\dots,k_{p+2}$.
\end{df}

\begin{df}[T-spline]\label{df: T-spline}
	We associate to each node $\node\in\nodes^d$ a multivariate B-spline, referred to as {T-spline}, defined as the product of the B-splines on the corresponding local index vectors, \begin{equation*}B_\node(x_1,\dots,x_d)\coloneqq B_{\kv_1(\node)}(x_1) \cdots B_{\kv_d(\node)}(x_d).\end{equation*}
\end{df}
We emphasize that given a structured $d$-dimensional T-mesh $\mesh^d$ with elements $\elements^d$, faces $\edges^d$, and nodes $\nodes^d$, the functions $B_\node$, for $\node \in \nodes^d$, are piecewise polynomials of degree $p$ in each direction over $\textsc{Bezier}(\elements^d)$ and $C^{p-1}$-smooth across all element interfaces $\edge \in \textsc{Bezier}(\edges^d)$. 

In the following, we use such (classical) T-splines based on tensor-product meshes to construct appropriate spline spaces over unstructured, two-dimensional T-meshes. While we use two-dimensional tensor-product splines as basis functions in structured regions, we use $k$-dimensional tensor-product splines to construct basis functions around extraordinary nodes of valence $k$.

\subsection{Unstructured T-spline spaces}\label{subsec: unstruct T-splines}

Throughout this subsection, we consider a two-dimensional T-mesh $\mesh$ as a control structure together with its B\'ezier mesh $\bmesh$. In order to simplify the construction, we suppose the following.
\begin{assu}\label{assu:extraordinarynodes}
	Let $\node\in\nodes$ be an extraordinary node of the mesh. We assume that all nodes $\node' \in \nodes$, with $\node' \neq \node$, that lie inside the {$(3p+1)/2$}-disk of $\node$, i.e., $\node'\in \nodes(\dc_{{(3p-1)/2}}(\node))$, are regular{, interior nodes}. Moreover, we assume that the $p$-disks of two extraordinary nodes do not overlap, i.e., for all $\node_1,\node_2\in\nodes$ with $\node_1\neq\node_2$ we have $\dc_{p}(\node_1) \cap \dc_{p}(\node_2) = \emptyset$.
\end{assu}
{The first condition implies that no knot line extensions extend into the $p$-disk of an extraordinary node and that all extraordinary nodes are sufficiently far away from the boundary. Thus, all elements in the $p$-disk of $\node$ are also elements of the B\'ezier mesh $\bmesh$, i.e., $\dc_{p}(\node) \subseteq \belements$. The second condition guarantees, that the supports of functions corresponding to different extraordinary nodes do not overlap.} These assumptions are not essential for the theory presented in this paper, but significantly facilitate technical details. 
It might be necessary to refine the initial mesh in order to guarantee that extraordinary nodes are sufficiently separated. A construction of analysis-suitable unstructured T-splines over more general meshes, which do not satisfy Assumption~\ref{assu:extraordinarynodes}, is presented in~\cite{WLQHZC:2021}. However, there the construction of functions around extraordinary nodes differs from the approach presented here.

Recall the unstructured spline space of Definition~\ref{df: spline space T-mesh}, which is piecewise polynomial of a certain degree $p$ over a given mesh and satisfies continuity conditions across edges as prescribed by a function $\kk$ defined on the edges of the mesh. 
To describe the desired spline space properly, we need the following definition.
\begin{df}[edge prolongation]
	For an edge $\edge\in\edges$, we define the \emph{edge prolongation} of $\edge$ as the set of edges that share a node but not an element with $\edge$, plus $\edge$ itself, i.e., 
	\begin{equation*}
		\ep(\edge) \sei \{\edge' \in \edges \,\vert\, \nodes(\edge)\cap \nodes (\edge') \neq \emptyset \wedge \elements(\edge)\cap \elements(\edge') = \emptyset \}\cup \{\edge\}
	\end{equation*}
	and for a set of edges $\cS \subseteq \edges$
	\begin{equation*}
		\ep(\cS) \sei \bigcup \{\ep(\edge)\,\vert\, \edge \in \cS\}.
	\end{equation*}
	Further, the \emph{edge prolongation of order} $i$ is defined as
	\begin{equation*}
		\ep^i(\cS) \sei \ep^{i-1}(\ep(\cS)) \text{ for }i > 1, \quad \ep^1(\cS) \sei \ep(\cS) \text{ for }i = 1.
	\end{equation*}
\end{df}

In the following, we aim to construct unstructured T-splines from the following space.
\begin{df}[unstructured spline space]\label{df: unstructured spline space}
	The unstructured spline space of degree $p$ over the mesh $\bmesh$ is given as 
	\begin{equation*}
		\splines^p(\bmesh,\kspecial),
	\end{equation*}
	where the function $\kspecial$ prescribing the continuity across edges of $\bedges$ takes the form
	\begin{equation*}
		\kspecial(\edge) = \begin{cases} 0 &\text{if $\edge$ is in $\ep^{p-1}(\edges(\node))$ for an extraordinary node $\node$,} \\ p-1 &\text{otherwise}\end{cases}%before it was $\ep^{p-2}$
	\end{equation*}
	for all $\edge\in\bedges$.
\end{df}
Due to $\edges\subseteq\bedges$, we have $\ep^{p-1}(\edges(\node)) \subset \bedges$, so the edges are all edges of the B\'ezier mesh as well and $\kspecial$ is thus well-defined for all edges. The choice of the continuity function $\kspecial$ for the space $\splines^p(\bmesh,\kspecial)$ corresponds to edges in $\mathrm{ep}^{(p-1)/2}(\edges(\node))$ having multiplicity $p$, when interpreted in classical T-spline notation, while all other edges have multiplicity $1$. Then the resulting B\'ezier mesh is $C^0$ across all edges in $\ep^{p-1}(\edges(\node))$ and $C^{p-1}$ across all other edges.

Providing a basis for the spline space of Definition~\ref{df: unstructured spline space} may not be feasible for general meshes. Thus, we construct a basis for a subspace. We split the construction as follows.
We define spline basis functions with support inside the $p$-disk 
of any extraordinary node via traces of multivariate tensor-product spline basis functions on a higher-dimensional structured grid (cf.\ Definition~\ref{df: T-spline} with $d>2$). Basis functions that do not have their support included in the $p$-disk of an extraordinary node
(i.e., basis functions that are associated with nodes outside the $(p+1)/2$-disk of any extraordinary node) are defined via standard bivariate T-splines (cf.\ Definition~\ref{df: T-spline} with $d=2$).

\begin{df}[T-splines over an unstructured T-mesh]\label{df: T-splines over an unstructured T-mesh}
	Let $\mesh=(\elements,\edges,\nodes)$ be a T-mesh. Let $\enodes \subset \nodes$ be the set of all extraordinary nodes. We define the set of anchors as
	\[
	\anchors = \nodes \setminus \bigcup_{\node \in \enodes {\cup \boundarynodes}} \{ \node' \in \nodes : \node' \subset \domain{\dc_{{(p+1)/2}}(\node)}\}.
	\]
	For each anchor $\node\in\anchors$, we define a basis function $B_\node$ as in Definition~\ref{df: structured T-spline}. For each extraordinary node $\node \in \enodes$, we define a set of functions $\cB_\node$ as in Definition~\ref{df: basis functions}. This yields a global set of functions
	\[
	\cB = \Big(\bigcup_{\node \in\enodes} \cB_\node\Big) \cup \{ B_\node : \node\in\anchors \}.
	\]
\end{df}
The functions given in Definition~\ref{df: basis functions} are well-defined due to Assumption~\ref{assu:extraordinarynodes}. {As a result of Definition~\ref{df: T-splines over an unstructured T-mesh}, there is a layer of elements around the boundary~$\partial\Omega$, where the spline space is not complete, cf.~Theorem~\ref{thm:polynomial-reproduction}.} The functions given in  Definition~\ref{df: structured T-spline} for all refinements of a regular mesh $\mesh_0$ are well-defined if the following is satisfied.
\begin{assu}\label{assu:disk-around-anchors}
	Let $\node\in\anchors$ be an anchor of the initial mesh $\mesh_0$. Then the elements in the $\frac{p+1}{2}$-disk $\dc_{(p+1)/2}(\node)$ form a structured submesh $\mesh_\node$, which possesses the parameter domain $\domain{\pullback{\mesh_\node}} = \left(-\frac{p+1}{2},\frac{p+1}{2}\right)^2$, where $\pullback{\node} = (0,0)^T$. Moreover, we assume that for any two anchors $\node,\node'\in\anchors$ the union $\dc_{(p+1)/2}(\node) \cup \dc_{(p+1)/2}(\node')$ forms a structured submesh as well.
\end{assu}
On the initial mesh the support of the function corresponding to $\node\in\anchors$ is given by $\dc_{(p+1)/2}(\node)$. The second part of the assumption guarantees that the union of supports of two functions is contained inside a structured submesh. Consequently, the same is true for any two functions on refined meshes.

The functions in $\cB$ are all linearly independent, see Theorem~\ref{thm: linear-idependence}, and they satisfy
\begin{equation}\label{eq:subspace}
	\mathrm{span} (\cB) \subseteq \splines^p(\bmesh,\kspecial).
\end{equation}
This construction is particularly useful in the context of refinement since refinement routines for structured grids are applicable. Note that we do not necessarily require equality in~\eqref{eq:subspace}. Although we have equality when the spaces are restricted to functions with support either completely outside the $p$-disk of an extraordinary node (by classical T-spline arguments) or completely inside (cf.~Proposition~\ref{prop: spline-space-extraordinary-node}), this does not need to be true in the transition domain. However, we do not explicitly demand equality in our construction and instead profit from its simplicity.

We emphasize that one can extend the definition of T-splines from unstructured T-meshes over planar partitions to mapped isogeometric functions. This is explained in more detail in~\ref{sec: geometry}. 

For all anchors inside the structured parts of the T-mesh we have the following.
\begin{df}[structured T-spline]\label{df: structured T-spline}
	Let $\node\in\anchors$ be an anchor of the mesh $\mesh$. Then the function 
	\[
	B_{\node} : \Omega \rightarrow \real
	\]
	has support in a structured submesh $\mesh_\node$ of $\mesh$ and is on its support defined as
	\[
	B_{\node} \circ \mapping{\elements_\node} = B_{\pullback{\node}},
	\]
	where $\mapping{\elements_\node}: \domain{\pullback{\mesh_\node}} \rightarrow \domain{\mesh_\node}$. Moreover, we have $\pullback{\mesh_\node}$ being a structured two-dimensional T-mesh and $\pullback{\node}$ being a node of that structured T-mesh.
\end{df}
Note that it follows from Assumption~\ref{assu:extraordinarynodes} and from the definition of the T-mesh refinement that such a submesh $\mesh_\node$ exists for all $\node\in\anchors$.

Next we define the basis functions around extraordinary nodes.
As a first step of the construction, we define a localized embedding of the unstructured mesh into a higher-dimensional structured grid, which is based on the following observation: 
the $p$-disk $\dc_{p}(\node)$ consists of $k$ non-overlapping structured submeshes $\{\mesh_i\}_{i=1}^k$,
in counterclockwise enumeration around $\node$,
that are separated by the prolongated edge-neighborhood of $\node$, given by $\ep^{p-1}(\edges(\node))$. That is, for $i \neq j$, we have that
\begin{equation}
	\domain{\mesh_i}\cap\domain{\mesh_j}=\emptyset \quad\text{and}\quad
	\overline{\domain{\mesh_i}} \cap \overline{\domain{\mesh_j}} \subset \overline{\ep^{p-1}(\edges(\node))}.\label{eq:structured-submeshes}
\end{equation}
Since by Assumption~\ref{assu:extraordinarynodes}, all nodes in $\nodes_1,\dots,\nodes_k$ except $\node$ are regular, each of the submeshes $\mesh_1,\dots,\mesh_k$ is formed by a regular grid of $p\times p$ elements.
\begin{df}[higher-dimensional embedding]\label{df: embedding}
	Let $\node$ be an extraordinary node with valence $v_{\node} = k$, and suppose that Assumption~\ref{assu:extraordinarynodes} holds.
	Then there exists an embedding $\emb^k\colon \domain{\dc_p(\node)}\to [0,p)^k\subset\real^k$ that maps the elements, edges, and nodes of $\dc_{p}(\node)$ onto faces, edges and nodes $(\elements^k,\edges^k,\nodes^k)$ of the $k$-dimensional structured mesh $\mesh^d$ with nodes in $\integer^k$ and elements in \begin{equation*}
		\bigl\{\bigtimes_{j=1}^k(z_j,z_j+1)\mid z_1,\dots,z_k\in\integer\bigr\}.
	\end{equation*}
	The embedding $\emb^k$ is characterized by the following choices. 
	\begin{itemize}
		\item $\emb^k(\node) = (0,\dots,0)\in\real^k$.
		\item For $j=0,\dots,k-1$, $\nodes_j$ is mapped to 
		$\bigl\{ae_j+be_{j+1}=(0,\dots,0,a,b,0,\dots,0)\mid a,b\in\{0,\dots,p\}\bigr\}$,
		with $e_j$ the $j$-th unit vector,  $\nodes_0=\nodes_k$ and $e_0=e_k$.
		\item For $j=1,\dots,k$, the edge $\edge_j\in\edges(\node)\cap\edges_j\cap\edges_{j-1}$ is mapped to 
		$\emb^k(\edge_j) = \bigl\{ae_j\mid 0<a<1\bigr\}$,
		with $\edges_0=\edges_k$.
		\item The remaining edges in the prolongated edge neighborhood $\edge_j\in\ep^{p-1}(\edges(\node))\cap\edges_j\cap\edges_{j-1}$, for $j=1,\dots,k$ and $\edges_0=\edges_k$, 
		are mapped to $\emb^k(\edge_j) = \bigl\{ae_j\mid i<a<i+1\bigr\}$,
		for $i\in\{1,\dots,p-1\}$ depending on $\edge_j$ such that $\emb^k$ is continuous.
		\item The elements $\element_j\in\elements_j$ are mapped to
		$\emb^k(\element_j) = \bigl\{ae_j+be_{j+1}\mid h<a<h+1,\ i<b<i+1\bigr\}$
		for $h,i\in\{1,\dots,p-1\}$ depending on $\element_j$ such that $\emb^k$ is continuous.
	\end{itemize}
	Together, this defines a unique map $\emb^k$ (up to rotated enumeration of the submeshes $\mesh_1,\dots,\mesh_k$) and we set
	$(\elements^k,\edges^k,\nodes^k) = \bigl(\emb^k(\dc_p(\node)), \emb^k(\edges(\dc_p(\node))), \emb^k(\nodes(\dc_p(\node)))\bigr).$ 
\end{df}

As a next step, we define a spline space on the $p$-disk $\dc_{p}(\node)$ of an extraordinary node $\node$ by employing the higher-dimensional representation of the mesh. Note that, as explained in Section~\ref{subsec: high-dim struct T-splines}, the higher-dimensional mesh $\mesh^k$ induced by $\integer^k$ allows for a classical definition of T-splines by means of a tensor-product construction. In this case the space is actually a standard, $k$-variate tensor-product B-spline space. For a clear distinction between bivariate splines associated with the unstructured mesh and multivariate splines associated with $\mesh^k$, we denote the latter T-spline basis functions corresponding to a node $\node=\{\nodeinRd\}$, $\nodeinRd\in\integer^k$, by $B_{\node}^k\colon\real^k \to \real$. The spline space in the original two-dimensional unstructured mesh is then constructed as follows.

\begin{df}[local T-spline space for an extraordinary node]\label{df: spline space extraordinary node}
	Let $r=(p-1)/2$ and let $\node$ be an extraordinary node and $\emb^k$ its embedding into the $k$-dimensional mesh $\mesh^k$. Then, we define the set of unstructured splines around $\node$ by
	\begin{equation}\label{eq:unstructuredSplineSpace}
		\mathcal{S}_\mathrm{u}(\node)\sei \mathrm{span}\big\{ B_{\node'}^k\circ \emb^k:\domain{\dc_{p}(\node)} \rightarrow \real,\quad \node' \in ([-r,r]\cap\integer)^k \big\}.
	\end{equation}
	That is, these functions are defined as traces of higher-dimensional structured spline functions.
\end{df}

\begin{rem}[Choice of the local T-spline space]
	We emphasize that the choice of the local construction defined in Definition~\ref{df: spline space extraordinary node} is not unique and other choices can be considered as well, e.g., the cubic unstructured T-spline constructions presented in~\cite{SSELBHS:2013,TosSH17,CWTLHKZ:2020,WLQHZC:2021}. However, the construction needs to be compatible with the refinement routine presented in Section~\ref{sec: Refinement} in the sense that spline spaces may be defined on any of the meshes generated with the refinement algorithm without losing linear independence. For the above construction, this is naturally fulfilled due to the fact that both the spline construction and the refinement routine are motivated by a higher-dimensional representation of the local mesh. 
\end{rem}

Due to the nature of the embedding defined in Definition~\ref{df: embedding}, the spline space $\mathcal{S}_u(\node)$ is exactly a space of spline functions supported in the $p$-disk of the extraordinary node $\node$ that 
\begin{itemize} 
	\item are $C^{p-1}$-continuous except for the prolongated edge-neighborhood $\mathrm{ep}^{p-1}(\edges(\node))$, where the functions are in general only $C^0$-continuous, and
	\item span those functions in $\splines^p(\bmesh,\kspecial)$ which have their support contained in the neighborhood $\domain{\dc_{p}(\node)}$ of the node $\node$. Due to the smoothness properties, the functions vanish and have vanishing derivatives up to order $p-1$ at the boundary of $\domain{\dc_{p}(\node)}$.
\end{itemize}
These properties are formalized in Proposition~\ref{prop: spline-space-extraordinary-node}. 
Note that for nodes in the unstructured two-dimensional mesh that are outside the $(p+1)/2$-disk of an extraordinary node, the corresponding spline basis function can be constructed as described in Section~\ref{subsec: high-dim struct T-splines} directly for $d=2$. For nodes inside the $(p+1)/2$-disk of an extraordinary node, we do not associate functions directly but rely on a basis construction derived from the embedding $\emb^k$. There exist several possibilities for the choice of corresponding basis functions that yield the space $\mathcal{S}_\mathrm{u}(\node)$ defined in~\eqref{eq:unstructuredSplineSpace}. It is easy to see that the functions $B_{\node'}^k\circ \emb^k$ for the entire index space $\node' \in ([-r,r]\cap\integer)^k$ are linearly dependent.
In the following, we present a choice of linearly independent basis functions of the space $\mathcal{S}_\mathrm{u}(\node)$. 

\begin{df}[basis functions near an extraordinary node]\label{df: basis functions}
	Let $r=(p-1)/2$. We consider the set of anchors
	\[
	\begin{array}{ll}
		\mathcal{I}^k_{r} =& \{ (i,j,-r,-r,\ldots,-r) : -r\leq i,j \leq r \} \\
		&\cup \{ (-r,i,j,-r,\ldots,-r) : -r\leq i,j \leq r \} \\
		&\cup \{ (-r,-r,i,j,\ldots,-r) : -r\leq i,j \leq r \} \\
		&\cup \dots \cup \{ (j,-r,\ldots,-r,i) : -r\leq i,j \leq r \} \subset ([-r,r]\cap\integer)^k.
	\end{array}
	\]
	Note that the cardinality of the set $\mathcal{I}^k_{r}$ is $k(4r^2+2r)+1=kp(p-1)+1$.
	Every node $\node$ in the set $\mathcal{I}^k_{r}$ is an anchor for a standard, $k$-variate basis function $B_{\node}^k$. 
	The set of corresponding basis functions on the two-dimensional unstructured mesh is then given by
	\begin{equation*}
		\cB_\node \sei \big\{ B^k_{\node}\circ \emb^k\;:\; \node \in \mathcal{I}^k_{r} \big\}
	\end{equation*}
	and the space of unstructured spline functions is defined by
	\begin{equation*}
		\mathcal{S}^*_\mathrm{u}(\node)\sei \mathrm{span}\{ b\;:\;b \in \cB_\node\}.
	\end{equation*}
\end{df}
We have that $\mathcal{S}^*_\mathrm{u}(\node)$ reproduces the classical $C^0$-continuous T-spline space based on knot and edge multiplicities of valence $p$ in the prolongated edge-neighborhood $\mathrm{ep}^{(p-1)/2}(\edges(\node))$ of the extraordinary node $\node$. Therefore, the set of basis functions as described in Definition~\ref{df: basis functions} is a natural choice for the spline space near the extraordinary node. 
\begin{prop}\label{prop: spline-space-extraordinary-node}
	Let $\node$ be an extraordinary node. Then, the set of functions as given in Definition~\ref{df: basis functions} forms a basis of the space $\mathcal{S}_\mathrm{u}(\node)$ as in Definition~\ref{df: spline space extraordinary node}. We moreover have
	\[
	\mathcal{S}^*_\mathrm{u}(\node) = \mathcal{S}_\mathrm{u}(\node) = \{ f \in \splines^p(\bmesh,\kspecial) : \mathrm{supp}(f) \subseteq \domain{\dc_{p}(\node)} \}.
	\]
\end{prop}

\begin{proof} 
	We trivially have
	\[
	\mathcal{S}^*_\mathrm{u}(\node) \subseteq \mathcal{S}_\mathrm{u}(\node) \subseteq \{ f \in \splines^p(\bmesh,\kspecial) : \mathrm{supp}(f) \subseteq \domain{\dc_{p}(\node)} \},
	\]
	since $\mathcal{I}^k_{r} \subset ([-r,r]\cap\integer)^k$ and all functions $B^k_{\node}$ are $C^{p-1}$ smooth piecewise polynomials and $\emb^k$ is $C^\infty$ except for $\mathrm{ep}^{p-1}(\edges(\node))$, where it is only $C^0$. Moreover, the condition $\mathrm{supp}(B^k_{\node} \circ \emb^k) \subseteq \domain{\dc_{p}(\node)}$ follows from $B^k_{\node} \subset [-p,p]^k$ for all $\node \in ([-r,r]\cap\integer)^k$ and from the definition of $\emb^k$.
	
	As described before in Definition~\ref{df: embedding}, let $\{\mesh_i\}_{i=1}^k$ be non-overlapping structured submeshes around the extraordinary node $\node$. Each of these submeshes also defines a subdomain $\domain{\mesh_i}\subset\Omega$. As before, $k=v_\node$ is the valence of the extraordinary node. 
	
	The dimension of the space 
	\begin{equation}\label{eq: splines with support in DDpN}
		\{ f \in \splines^p(\bmesh,\kspecial) : \mathrm{supp}(f) \subseteq \domain{\dc_{p}(\node)} \}
	\end{equation}
	equals $kp(p-1)+1$, which can be checked using a simple counting argument following the definition of standard $C^0$ multi-patch B-splines as, e.g., in~\cite[Section~3.3]{BS:2016}.
	
	The last step consists in showing that the functions in $\cB_\node$ are linearly independent. Due to $|\cB_\node| = kp(p-1)+1$, these functions then form a basis of the space \eqref{eq: splines with support in DDpN} above.
	
	To prove this, let $\cB_\node = \{b_\ell\}_{\ell = 1}^{kp(p-1) + 1}$ be the set of functions that are constructed as described in Definition~\ref{df: basis functions}. Each of these functions has support
	\begin{itemize}
		\item on $i \times j$ elements ($1 \leq i,\,j \leq p$) on one of the submeshes and whose boundary includes $\node$ and
		\item on $i \times 1$ or $1 \times j$ or $1 \times 1$ elements on all the remaining submeshes.
	\end{itemize}
	Based on the functions $\cB_\node$, we construct a set of functions $\tilde\cB_\node = \{\tilde b_\ell\}_{\ell = 1}^{kp(p-1) + 1}$ with the same cardinality, for which the linear independence is more apparent due to support arguments. To this end, we sketch the construction considering three different cases below.
	
	\paragraph{Case 1: function centered at $\node$.}
	The function $\tilde b_1 = b_1$, which has support on exactly one element on each of the submeshes and corresponds to the index $(-r,-r,\ldots,-r) \in \mathcal{I}^k_{r}$, is the only function in the modified set of functions that is supported on more than two submeshes and therefore linearly independent to the functions of Case 2 and 3.
	
	\paragraph{Case 2: functions associated to nodes on $\mathrm{ep}^{r}(\edges(\node))$.}
	As a second case, we consider the functions $\{b_\ell\}_{\ell = 2}^{k(p-1)+1}$ that have support on exactly one element on all but two submeshes and have support on $i \times 1$ and $1 \times i$ elements ($2 \leq i \leq p$) on the remaining two submeshes, respectively. In the index notation, these functions are characterized by the indices $(i,-r,-r,\ldots,-r)$, $(-r,i,-r,\ldots,-r),\,\ldots,\,(-r,-r,-r,\ldots,i)\in \mathcal{I}^k_{r}$ for $-r<i\leq r$.
	
	Due to the definition of these functions and with the function from Case~1, we can subtract an appropriate multiple of $b_1$ to obtain functions that are only supported on the two submeshes with support on $i \times 1$ and $1 \times i$ elements, respectively. We denote them with $\{\tilde{b}_\ell\}_{\ell = 2}^{k(p-1)+1}$. These functions are linearly independent to each other since they all have different supports that cannot be obtained by the combination of the supports of other functions. Further, the functions from Case~2 are all linearly independent to the functions from Case 3 since all functions from Case 3 are only supported on one submesh.
	
	\paragraph{Case 3: remaining functions.}
	The remaining functions $\{b_\ell\}_{\ell = k(p-1)+2}^{kp(p-1)+1}$ have support on $i \times j$ elements ($2 \leq i,\,j \leq p$) on one (main) submesh, on $i \times 1$ and $1 \times j$ elements on the two edge-neighboring submeshes, respectively, and on one element on the remaining submeshes. These functions correspond to the remaining indices $(i,j,-r,\ldots,-r)$, $(-r,i,j,\ldots,-r),\,\ldots \in\mathcal{I}^k_{r}$ with $-r<i,j\leq r$ and can be modified to functions $\{\tilde{b}_\ell\}_{\ell = k(p-1)+2}^{kp(p-1)+1}$ whose supports lie within the $i \times j$ elements of only one submesh. 
	
	For a particular functions $b \in \{{b}_\ell\}_{\ell = k(p-1)+2}^{kp(p-1)+1}$, we achieve this by subtracting appropriate multiples of two functions in $\{\tilde b_\ell\}_{\ell = 2}^{k(p-1)+1}$, namely those with overlap of $1 \times i$ and $j \times 1$ elements with the main submesh, respectively. To compensate for the overlap of these two functions, a multiple of $b_1$ has to be added to obtain functions that are only supported within one submesh. As above, the resulting functions all have different supports that cannot be obtained by an appropriate combination of the supports of the remaining functions. 
	
	\bigskip\noindent
	In summary, the functions in $\tilde\cB_\node$ are linearly independent. Since they are constructed from the set $\cB_\node$ with equal size, also the functions in $\cB_\node$ are linearly independent, which completes the proof.
\end{proof}

\section{Refinement}
\label{sec: Refinement}
The refinement algorithm makes use of the following key features:
\begin{enumerate}
	\item an edge subdivision routine that allows for I-nodes, see Algorithm~\ref{alg: subdiv},
	\item refinement levels of edges,
	\item an edge neighborhood based on the mesh metric from Definition~\ref{df: mesh metric}, 
	\item direction indices, i.e., integers associated to each edge as explained in Definition~\ref{df: direction index} below.
\end{enumerate}

\begin{algorithm}
	\caption{Subdivision of an edge}\label{alg: subdiv}
	\begin{algorithmic}[1] 
		\Procedure{subdiv}{$\mesh,\edge=[a,b]$} \Comment{the mesh and the edge to be subdivided}
		\State $\midp(\edge)\gets \frac{a+b}2$;\enspace $\edge_1\gets[a,\midp(\edge)]$;\enspace $\edge_2\gets[\midp(\edge),b]$\label{alg line: edge children}
		\State $\nodes\gets \nodes\cup\{\midp(\edge)\}$\Comment{update nodes}
		\State $\edges\gets \edges\cup\{\edge_1,\edge_2\}$\Comment{update edges}
		\State mark $\edge$ as inactive/refined
		\ForAll {$\Q\in\elements$ neighboring $\edge$}\Comment{addresses at most two elements}
		\State $\edge'\gets$ opposite edge of $\edge$ in $\Q$
		\If {$\edge'$ has been subdivided} 
		\State $\tilde \edge\gets$ new edge connecting the midpoints of $\edge$ and $\edge'$\label{alg line: other new edge}
		\State $\{\Q_1,\Q_2\}\gets$ children of $\Q$ being subdivided by $\tilde \edge$
		\State $\edges\gets\edges\cup\{\tilde \edge\}$
		\State $\elements\gets\elements\cup\{\Q_1,\Q_2\}$
		\State mark $\Q$ as inactive/refined
		\EndIf
		\EndFor
		\State \textbf{return} $\mesh$
		\EndProcedure
	\end{algorithmic}
\end{algorithm}

\begin{df}[refinement level]\label{df: refinement-level}
	All edges of the initial mesh have the refinement level zero, written $\ell(\edge)=0$ for all $\edge\in\edges$.
	The levels of new edges created during refinement obey the following rules.
	\begin{enumerate}
		\item The children $\edge_1,\edge_2$ of an edge $\edge$ as in Algorithm~\ref{alg: subdiv}, line~\ref{alg line: edge children}, have the level of their parent, increased by 1, i.e., $\ell(\edge_1)=\ell(\edge_2)=\ell(\edge)+1$~.
		\item Other new edges such as $\tilde \edge$ in Algorithm~\ref{alg: subdiv}, line~\ref{alg line: other new edge}, have the same level as their opposite edges.
	\end{enumerate}
\end{df}

\begin{df}[edge neighborhood]\label{df: neighb}
	Given a mesh $\mesh$, we define for each edge $\edge\in\edges$ the neighborhood
	\[\neighb(\mesh,\edge) = \bigl\{\edge'\in\edges\mid\dist(\edge,\edge')\le \tfrac{p+1}2 \cdot 2^{-\ell(\edge)}\bigr\},\]
	where $\dist(\edge,\edge')$ is the abbreviation for $\dist(\midp(\edge),\midp(\edge'))$.
\end{df}

\begin{figure}[t]
	\centering
	\includegraphics{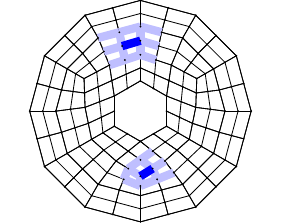}\quad
	\includegraphics{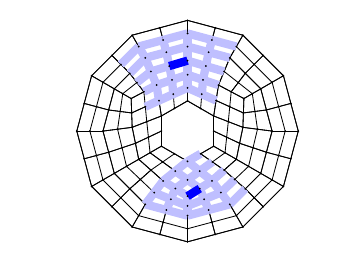}\quad
	\includegraphics{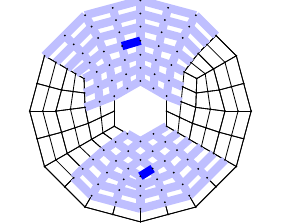}
	\caption{Examples for edge neighborhoods for $p=1$ (left) and $p=3$ (center) and $p=5$ (right). }
\end{figure}

The main difficulty for T-splines on unstructured meshes is that there is no global notion of directions in the sense of horizontal and vertical edges.
However, we can assign edge directions on each element and require  consistency of these directions across the mesh. This is realized by integer-valued direction indices, which are defined below.

\begin{df}[direction index]\label{df: direction index}
	Given a T-mesh $\mesh$, the \emph{direction index} $\di(\edge)$ of any edge $\edge\in\edges$ is defined by a mapping $\di:\edges\to\nat$ such that
	\begin{equation}
		\label{eq: di criterion}
		\forall\ \Q\in\elements,\ \edge_1,\edge_2\in\edges(\Q):\quad \di(\edge_1)=\di(\edge_2)\Leftrightarrow \overline{\edge_1}\cap \overline{\edge_2}=\emptyset .
	\end{equation}
	This is, edges opposite to each other have identical direction indices, whereas neighboring edges in any element have distinct direction indices. Note, however, that $\di(\edge_1)=\di(\edge_2)$ \emph{and} $\overline{\edge_1}\cap \overline{\edge_2}\neq\emptyset$ is allowed if there is no common element $\Q\in\elements$ such that $\edges(\Q)$ contains both $\edge_1$ and $\edge_2$, see Figure~\ref{f:directionInd} for an illustration. {We emphasize that the concept of direction indices is motivated by the different coordinate directions in higher dimensions. That is, if the unstructured mesh is obtained from traces of a higher-dimensional mesh, a direction index of an edge (locally) corresponds to the coordinate direction to which the corresponding edge in the higher-dimensional mesh is parallel.}
	
	Any map $\di\colon\edges\to\nat$ satisfying \eqref{eq: di criterion} is called \emph{direction labeling}, and its value $\di(\edge)$ at $\edge\in\edges$ is called  direction index of $\edge$.
	
	The direction indices of new edges created during refinement are defined through the rules
	\begin{enumerate}
		\item The children $\edge_1,\edge_2$ of an edge $\edge$ as in Algorithm~\ref{alg: subdiv}, line~\ref{alg line: edge children}, have the same direction index as their parent.
		\item Other new edges such as $\tilde \edge$ in Algorithm~\ref{alg: subdiv}, line~\ref{alg line: other new edge}, 
		have the same direction index as their opposite edges.
	\end{enumerate}
	
\end{df}

\begin{figure}[t]
	\centering
	\begin{minipage}[b]{.45\textwidth}
		\centering
		\includegraphics[width=.9\textwidth]{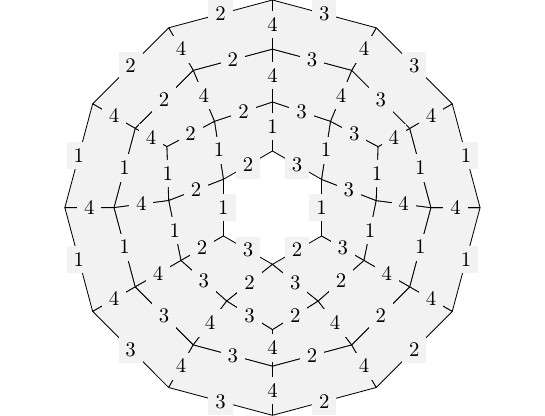}
		\caption{A direction labeling for the mesh from Figure~\ref{fig: ring mesh}.}\label{f:directionInd}
	\end{minipage}\hspace{.05\textwidth}%
	\begin{minipage}[b]{.45\textwidth}
		\centering
		\includegraphics[width=0.5\textwidth]{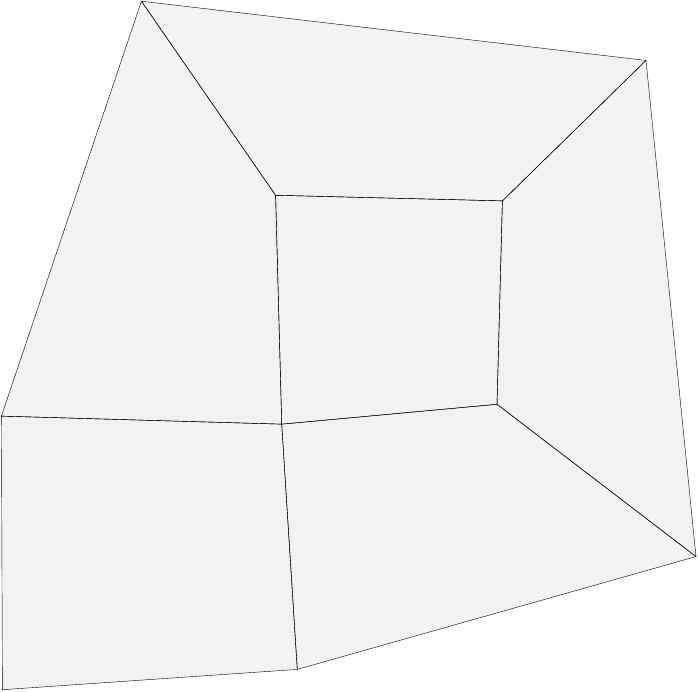}
		\caption{Example for a totally unstructured mesh.}\label{fig: tu mesh}
	\end{minipage}
\end{figure}

For practical purposes, it is important to find a configuration of direction indices with a small number of distinct indices for a given unstructured mesh. We suggest the following approach: for any two edges $\edge$ and $\edge'$, we write $\edge\sim\edge'$ if $\edge$ and $\edge'$ are opposite edges of a common neighboring cell. 
In addition, we require that the relation $\sim$ is reflexive and transitive, such that $\edge\sim\edge'$ is equivalent to the fact that the rule \eqref{eq: di criterion} enforces $\di(\edge)=\di(\edge')$. 
We draw an undirected graph, where each node represents an equivalence class $[\edge]$ (with $[\edge]=[\edge'] \Leftrightarrow \edge\sim\edge'$), 
and two nodes $[\edge]$, $[\edge']$ are connected if $\edge$ and $\edge'$ are adjacent edges of a common neighboring cell.
This construction reduces our problem to a standard graph coloring problem, for which we can apply existing approaches, see, e.g., \cite{Farzaneh:2022}. In the context of totally unstructured meshes, this approach is presented in more detail in \ref{sec: other meshes}.

In the following, we assume that there exists a direction labeling. This is an important assumption, 
as there exist meshes for which no direction labeling is possible, since the graph constructed above may have self-connected nodes. 
According to \cite[Definition~7.2.2]{Morgenstern:2017}, these meshes are called \emph{totally unstructured}, see Figure~\ref{fig: tu mesh} for an illustration and~\ref{sec: other meshes} for generalized direction indices on these meshes.

Our refinement algorithm is described in Algorithm~\ref{alg: ref} and example outputs are shown in Figure~\ref{fig: refinement example}.
\begin{algorithm}
	\caption{Refinement algorithm}\label{alg: ref}
	\begin{algorithmic}[1] 
		\Procedure{refine}{$\mesh,\edge$} 
		\While {$\Exists \edge'\in\neighb(\mesh,\edge):\enspace\ell(\edge')<\ell(\edge)$ 
			\\or $\Exists \edge'\in\neighb(\mesh,\edge):\enspace\ell(\edge')=\ell(\edge) \wedge \di(\edge')<\di(\edge)$}
		\ForAll {$\edge'\in\neighb(\mesh,\edge)$}
		\If {$\ell(\edge')<\ell(\edge)$ } 
		\State $\mesh\gets\refine(\mesh,\edge')$
		\ElsIf {$\ell(\edge')=\ell(\edge) \wedge \di(\edge')<\di(\edge)$}
		\State $\mesh\gets\refine(\mesh,\edge')$
		\EndIf
		\EndFor
		\EndWhile
		\State $\mesh\gets\textsc{subdiv}(\mesh,\edge)$
		\State \textbf{return} $\mesh$
		\EndProcedure
	\end{algorithmic}
\end{algorithm}

\begin{figure}[t]
	\centering
	\includegraphics[width=.45\textwidth]{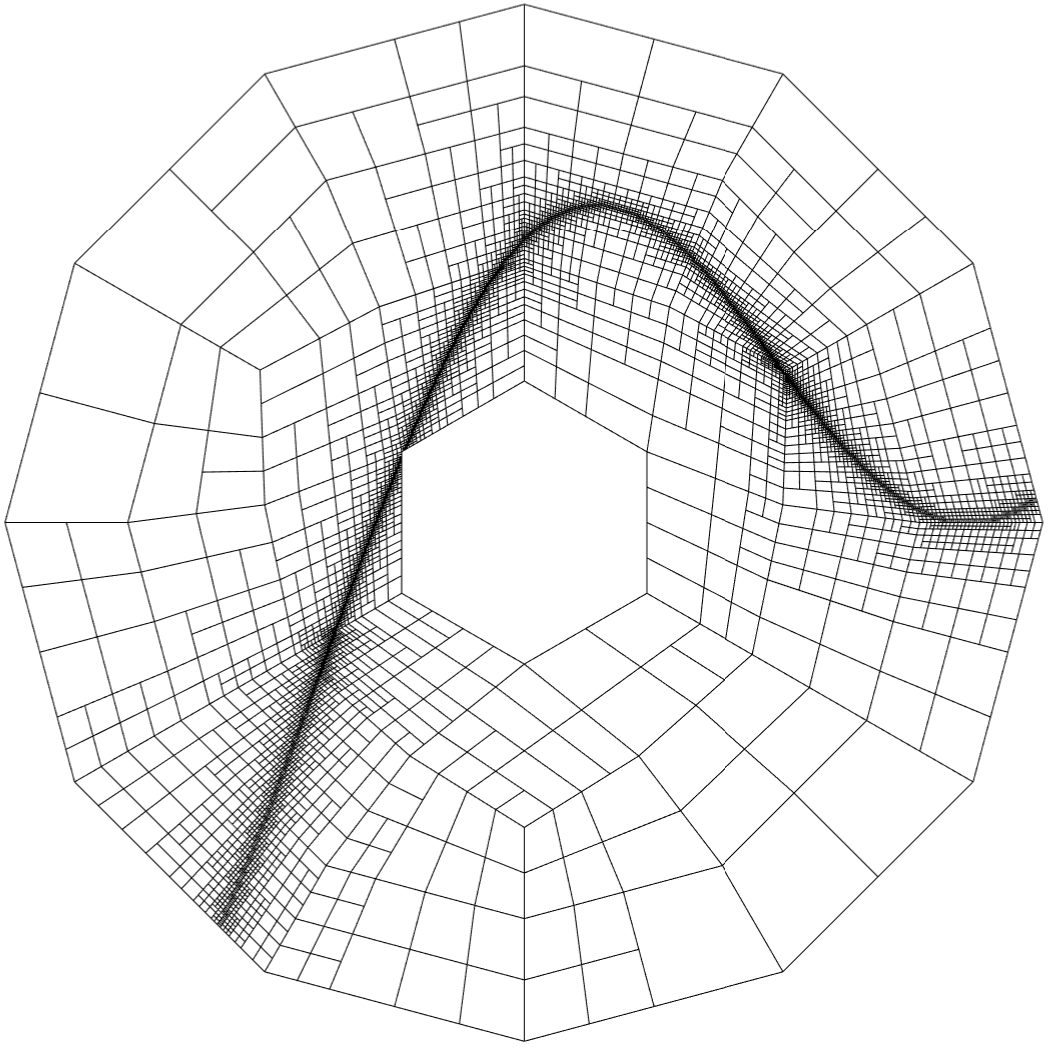}\hspace{.05\textwidth}%
	\includegraphics[width=.45\textwidth]{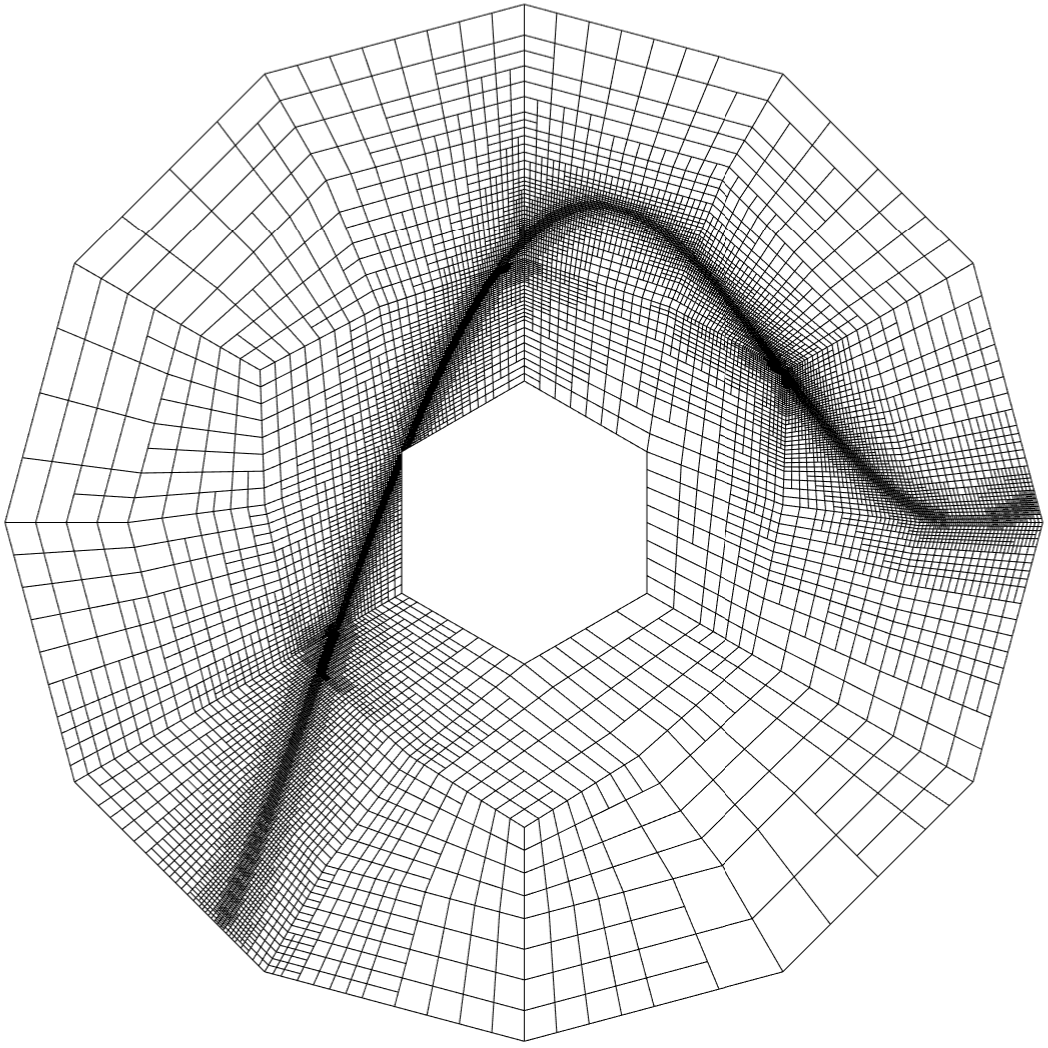}
	\caption{Refinement examples for $p=1$ (left) and $p=3$ (right), where the mesh is refined towards a given curve.}
	\label{fig: refinement example}
\end{figure}

\section{Properties of T-meshes and T-splines}
\label{sec: mesh properties}

By construction, the basis functions in~$\cB$, as given in Definition~\ref{df: T-splines over an unstructured T-mesh}, have local support and satisfy~\eqref{eq:subspace}, i.e.,
\begin{equation*}
	\mathrm{span} (\cB) \subseteq \splines^p(\bmesh,\kspecial).
\end{equation*}
Moreover, we have the following.
\begin{theorem}\label{thm:polynomial-reproduction}
	For all elements $\element\in\elements$ that are sufficiently far away from the boundary of the mesh, we have that 
	\begin{equation*}
		\mathrm{span} (\cB) |_{\element} = \mathbb{P}^{(p,p)}(\element).
	\end{equation*}
\end{theorem}
\begin{proof}
	Recall the definition of $\cB$ as
	\[
	\cB = \Big(\bigcup_{\node \in\enodes} \cB_\node\Big) \cup \{ B_\node : \node\in\anchors \}.
	\]
	For each extraordinary node $\node\in \enodes$, the support of $\cB_\node$ is given by $\domain{\dc_{p}(\node)}$. Due to Assumption~\ref{assu:extraordinarynodes}, the $(3p-1)/2$-disk $\dc_{(3p-1)/2}(\node)$ around an extraordinary node $\node$ is composed of $k$ non-overlapping structured submeshes $\{\mesh_i^\ast\}_{i=1}^k$. Note that the meshes $\mesh_i^\ast$ contain the submeshes $\mesh_i$ in~\eqref{eq:structured-submeshes}.
	
	We assume that $\element\in\elements$ is sufficiently far away from the boundary $\partial\Omega$, i.e., $\element\notin \dc_{p}(\node)$ for all boundary nodes $\node\in\boundarynodes$. 
	We distinguish two different cases. We either have
	\begin{enumerate}
		\item $\element \in \dc_{p}(\node)$ for exactly one extraordinary node $\node\in \enodes$, or
		\item $\element \notin \dc_{p}(\node)$ for all extraordinary nodes $\node\in \enodes$.
	\end{enumerate}
	Since the $p$-disks $\dc_{p}(\node)$ of all extraordinary nodes $\node\in \enodes$ are disjoint, these are the only possible cases.
	
	\textit{Case 1:} Let $\element \in \dc_{p}(\node)$ for an extraordinary node $\node$. Let $\mesh_j^\ast$ be the submesh containing $\element$. On $\mesh_j^\ast$ the functions in 
	\[
	\cB_\node \cup \{ B_{\node'} : \node' \in \nodes(\dc_{(3p-1)/2}(\node)) \}
	\]
	form a standard tensor-product basis, since all nodes in $\nodes(\dc_{(3p-1)/2}(\node))$ are regular. Thus, they span all tensor-product polynomials of degree $p$ on $\element$.
	
	\textit{Case 2:} The polynomial reproduction on $\element$ follows from the standard T-spline construction, since all functions which have support on $\element$ are standard T-splines from $\{ B_\node : \node\in\anchors \}$.
	
	This concludes the proof.
\end{proof}
In Figure~\ref{fig:extr-node_neighborhood}, we visualize the construction of the spaces around an extraordinary node. In this example, we have $p=3$. In the left part, we show an example of a mesh around an extraordinary node $\node$ (marked with a black star). We assume that the depicted part of a mesh is sufficiently far away from the boundary. The green region marks the $p$-disk around the node, which corresponds to the union of supports of functions in $\cB_\node$. Similarly, the region marked in red represents the support of $\cB_{\node'}$ for another extraordinary node $\node'$ (marked with a gray star). The two extraordinary nodes satisfy that their corresponding $p$-disks (marked in green and red, respectively) do not overlap. Moreover, all nodes inside the $(3p+1)/2$-disk are regular. Thus, Assumption~\ref{assu:extraordinarynodes} is satisfied. However, nodes on the boundary of the $(3p+1)/2$-disk may be T-nodes, as indicated in the figure. The black and gray nodes represent the anchors of classical tensor-product T-spline basis functions which can be defined on a regular two-dimensional submesh. The construction of functions in $\cB_\node$ that have support completely contained in the green region is visualized in the right part of Figure~\ref{fig:extr-node_neighborhood}. More precisely, each black node represents an anchor from $\mathcal{I}_r^k$ as in Definition~\ref{df: basis functions}, where $r=(p-1)/2$ and $k=3$. The supports of the corresponding $k$-dimensional basis functions are all contained in the green region, when intersected with the coordinate planes.
\begin{figure}[t]
	\centering
	\includegraphics[width=.4\textwidth]{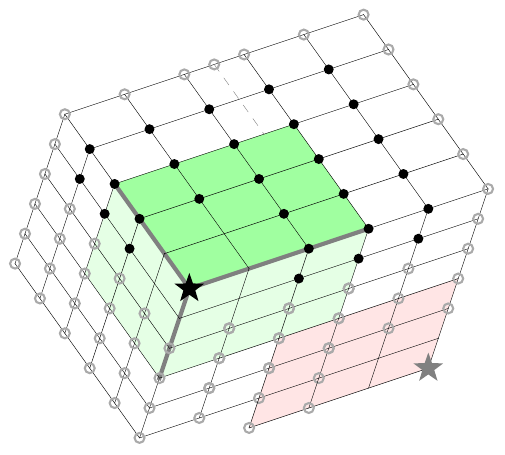}\qquad
	\includegraphics[width=.4\textwidth]{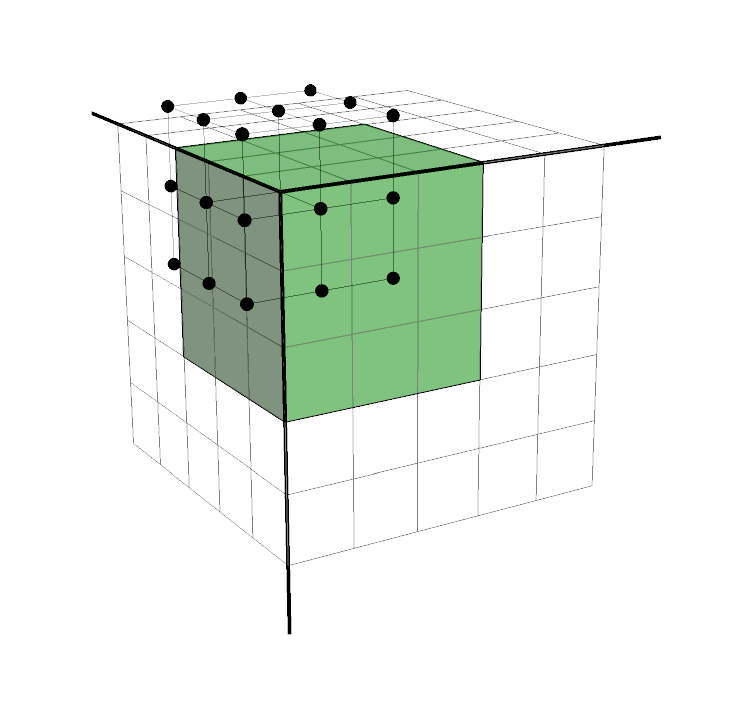}
	\caption{{Left: An example of a neighborhood of an extraordinary node $\node$ for degree $p=3$. The green region marks the support of $\cB_\node$, while the red region marks the support of $\cB_{\node'}$ for another extraordinary node $\node'$. The anchors are marked with circles, where the full, black circles are denoting those anchors that correspond to functions having support in the darker green region. The thick lines mark the edges of $C^0$-smoothness for functions in $\cB_\node$. Right: The higher-dimensional embedding and relevant basis functions, as in Definition~\ref{df: basis functions}. The anchors of basis functions are marked with black bullets. Their joint support, intersected with the coordinate planes, is marked in green.}}
	\label{fig:extr-node_neighborhood}
\end{figure}

{Note that polynomial reproduction is only guaranteed for elements that are sufficiently far away from the boundary. With minor modifications, generalizing B-splines on open knot vectors, the space can be adapted to obtain interpolatory B-splines at the boundary, which reproduce polynomials on all elements.} 

In the following subsections, we study properties of the proposed refinement algorithm. First, we show that the effect of the refinement algorithm is always local and of linear complexity and after that we show that the resulting functions remain linearly independent.

\subsection{Linear complexity}
In this subsection, we present two theorems that adress the complexity of the refinement procedure.
Theorem~\ref{thm: locality} states the existence of a uniform upper bound $C$ on the distance between an edge called for refinement and any edge that is produced by the recursive refinement procedure.
Theorem~\ref{thm: complexity} states that there is a uniform upper bound $C$ on the ratio between the accumulated numbers of produced and marked edges over the whole refinement history. 

In order to indicate that this is not always given, we present here an example of a refinement strategy that does not have linear complexity, a scenario which -- for computational reasons -- should be avoided.
Consider the following intuitive set of refinement rules for bilinear T-splines:
\newenvironment{inlinetikz}{\begin{tikzpicture}[scale=.35, baseline=2pt]}{\end{tikzpicture}}
\begin{enumerate}
	\item We use edge bisections as in Algorithm~\ref{alg: subdiv}. 
	\item Any cell may contain at most one T-junction (%
	\begin{inlinetikz} \draw (0,0) grid (2,1) (0,.5)--(1,.5) (.5,0)--(.5,.5);  \end{inlinetikz} is ok, 
	\begin{inlinetikz} \draw (0,0) grid (2,1) (0,.5)--(1,.5) (.5,0)--(.5,.5) (.75,0)--(.75,.5);  \end{inlinetikz} is not ok).
	\item If a refinement of an edge causes more than one T-junction on a neighbor element $\element$, then $\element$ is refined in direction of the previously existing T-junction (this is, the opposite edge of the old T-junction is bisected such that 
	\begin{inlinetikz} \draw (0,0) grid (2,1) (0,.5)--(1,.5) (.5,0)--(.5,.5) (.75,0)--(.75,.5);  \end{inlinetikz} becomes 
	\begin{inlinetikz} \draw (0,0) grid (2,1) (0,.5)--(1,.5) (.5,0)--(.5,1) (.75,0)--(.75,.5);  \end{inlinetikz}).
\end{enumerate}

We start with an initial mesh that is structured and contains a single cell, and
we apply $2n\in\mathbb N$ bisections on horizontal edges, obtaining $n+1$ cells and $3n+4$ edges.
\newcommand{\examplemesh}[1]{%
	\begin{tikzpicture}[scale=2, baseline=2.5em]
		\pgfmathsetmacro{\n}{#1}
		\draw (0,0) grid[xstep=.2] +(1.6,1);
		\foreach \a in {1,...,\n} \draw (0,1-.5^\a)--(.2*\n-.2*\a,1-.5^\a);
		\foreach \a in {0,...,8} \foreach \b in {0,1} \fill (.2*\a,\b) circle (.5pt);
		\foreach \a in {1,...,\n} \foreach \b in {\a,...,\n} \fill (.2*\a-.2,1-.5*2^\a*.5^\b) circle (.5pt);
	\end{tikzpicture}%
}
\[ \begin{tikzpicture}[scale=2, baseline=2.5em]  \draw (0,0) rectangle (1.6,1);
	\foreach \a in {0,1.6} \foreach \b in {0,1} \fill (\a,\b) circle (.5pt); \end{tikzpicture}%
\quad\to\quad 
\begin{tikzpicture}[scale=2, baseline=2.5em]  \draw (0,0) grid[xstep=.2] (1.6,1);
	\foreach \a in {0,...,8} \foreach \b in {0,1} \fill (.2*\a,\b) circle (.5pt); \end{tikzpicture}%
\]
For $k=1,\dots,n+2$, we mark the vertical edge in the upper left corner for bisection, increasing the number of edges by $2k-1$ in each step. 
\begin{gather*}	 
	\examplemesh1 \quad\to\quad \examplemesh2\quad\to\quad \examplemesh3
	\\[1em]
	\quad\to\quad \examplemesh4
	\to\dots\to \examplemesh9 \quad
\end{gather*}
In total we have marked $\#\edges_\mathrm M = 3n+2$ edges for refinement, increasing the number of edges by 
\begin{equation}
	\#\edges_{2n+1}-\#\edges_0 = 3n+\sum_{k=1}^{n+2}(2k-1)=3n+(n+2)^2 = n^2+7n+4.
\end{equation}
For any given constant $C>0$, we can choose $n=\lceil 3C\rceil$ in order to generate a sequence of refinements such that the ratio of generated and marked edges is 
\begin{equation}
	\frac{\#\edges_{2n+1}-\#\edges_0}{\#\edges_\mathrm M} = \frac{n^2+7n+4}{3n+2} 
	>\frac n3\ge C.
\end{equation}
This refinement strategy thus does not have linear complexity in the sense of Theorem~\ref{thm: complexity}.

\begin{theorem}[locality of the refinement]\label{thm: locality}%
	Let $\edge$ be an edge in the mesh $\mesh$, then any new edge $\edge'$ in $\refine(\mesh,\edge)\setminus\mesh$ satisfies
	\begin{equation*}
		\dist(\edge,\edge') < \Cref{thm: locality}\,2^{-\ell(\edge')}
		\quad\text{and}\quad \ell(\edge')\le\ell(\edge)+1,
	\end{equation*}
	where $\dist(\edge,\edge')$ is the abbreviation for $\dist(\midp(\edge),\midp(\edge'))$,
	$K=\#\di(\{\edge_0,\dots,\edge_J\})$ is the number of distinct direction indices occuring in the neighborhoods between $\edge$ and $\edge'$,
	and
	\[\Cref{thm: locality} = \tfrac12 + \tfrac{2}{2^{1/K}(2^{1/K}-1)}\cdot (p+1) < \cref{thm: locality} K(p+1),\]
	where $\cref{thm: locality}$ does not depend on the mesh $\mesh$, $K$, or $p$.
\end{theorem}
Theorem~\ref{thm: locality} is a very common result in the theory of mesh-adaptive Galerkin schemes, see, e.g., \cite[Lemmas 3.1.11 and 5.4.14]{Morgenstern:2017} for Hierarchical B-splines and T-splines on structured meshes, or \cite[Lemma 18]{Nochetto:Veeser:2012} for the Newest Vertex Bisection on triangular meshes.
\begin{proof}
	The existence of the edge $\edge'$ in $\refine(\mesh,\edge)\setminus\mesh$ means that Algorithm~\ref{alg: ref} has subdivided a sequence of edges $\edge=\edge_J,\edge_{J-1},\dots,\edge_0$ with $\edge'$ being a child of $\edge_0$, thus $\ell(\edge')=\ell(\edge_0)+1$, and with
	$\edge_{j-1}\in\neighb(\mesh,\edge_j)$ and either $\ell(\edge_{j-1})<\ell(\edge_j)$ or $\ell(\edge_{j-1})=\ell(\edge_j)\wedge\di(\edge_{j-1})<\di(\edge_j)$ for $j=1,\dots,J$.
	
	Let $K=\#\di(\{\edge_0,\dots,\edge_J\})$ be the number of distinct direction indices in this sequence. Then, we have $\ell(\edge_{j-K})<\ell(\edge_j)$, or equivalently,
	$\ell(\edge_{j-K})+1\le\ell(\edge_j)$, for any $j\in\{K,\dots,J\}$. 
	Repeated application yields for any $j\in\{{K},\dots,J\}$
	\begin{equation}\label{eq: thm locality: lEj ge lE0 + j/K}
		\ell(\edge_j)\ge \ell(\edge_{j-K})+1
		\ge \dots \ge \ell(\edge_{j-\lfloor j/K\rfloor\cdot K})+\lfloor \tfrac jK\rfloor
		\ge \ell(\edge_0)+\tfrac {j-K+1}K .
	\end{equation}
	Since $\dist$ is a metric, we have the triangle inequality
	\begin{equation}\label{eq: thm locality: dist = dist + dist}
		\dist(\edge',\edge) \le \dist(\edge',\edge_0) +\dist(\edge_0,\edge) 
	\end{equation}
	with
	\begin{equation}\label{eq: thm locality: dist details}
		\dist(\edge',\edge_0) = \tfrac12\diam(\edge') = 2^{-\ell(\edge')-1}, \qquad
		\dist(\edge_0,\edge) \le \sum_{j=1}^J\dist(\edge_{j-1},\edge_j) .
	\end{equation}
	Since $\edge_{j-1}\in\neighb(\mesh,\edge_j)$, we get by construction that 
	$\dist(\edge_{j-1},\edge_j)\le p\cdot 2^{-\ell(\edge_j)}$ and hence
	\begin{align*}
		\sum_{j=1}^J\dist(\edge_{j-1},\edge_j)
		&\le \sum_{j=1}^J\tfrac{p+1}2\cdot 2^{-\ell(\edge_j)}
		\stackrel{\eqref{eq: thm locality: lEj ge lE0 + j/K}}\le  \sum_{j=1}^J(p+1)\cdot 2^{-1-\ell(\edge_0)-(j-K+1)/K }
		\\&= (p+1)\cdot2^{-1-\ell(\edge_0)+(K-1)/K}\sum_{j=1}^J (2^{-1/K})^j
		< (p+1)\cdot2^{-(\ell(\edge_0)+1)+(K-1)/K}\frac{2^{-1/K}}{1-2^{-1/K}}
		\\&= (p+1)\cdot2^{-\ell(\edge')+(K-1)/K}\frac{2^{1/K}}{2^{2/K}-2^{1/K}}
		= (p+1)\cdot2^{-\ell(\edge')}\frac2{2^{1/K}(2^{1/K}-1)}
	\end{align*}
	The combination with \eqref{eq: thm locality: dist = dist + dist} and \eqref{eq: thm locality: dist details} concludes the distance estimate.
	The application of \eqref{eq: thm locality: lEj ge lE0 + j/K} to the case $j=J\ge0$, together with $\ell(\edge_0)=\ell(\edge')-1$ from above, yields
	\begin{equation*}
		\ell(\edge)\ge\ell(\edge_0) + \tfrac{J-K+1}K = \ell(\edge')-2 + \tfrac{J+1}K,
	\end{equation*}
	hence $\ell(\edge')\le \ell(\edge) +2 -  \tfrac{J+1}K$. Since $\tfrac{J+1}K>0$ and $\ell(\edge'), \ell(\edge)$ are integer numbers, we get \mbox{$\ell(\edge')\le \ell(\edge) +1$}. 
	{Thus, we obtain 
		\[
		\dist(\edge,\edge') < \Cref{thm: locality}\,2^{-\ell(\edge')}
		\]
		with 
		\[
		\Cref{thm: locality} = \tfrac12 + \tfrac{2}{2^{1/K}(2^{1/K}-1)}\cdot (p+1).
		\]
		To conclude the proof, we show that 
		\[
		\tfrac{\Cref{thm: locality}-1/2}{K(p+1)} = \tfrac{2}{K 2^{1/K}(2^{1/K}-1)} \rightarrow \tfrac{2}{\ln(2)}
		\]
		for $K\rightarrow \infty$. To do so, we replace $K$ by $x=1/K$ and obtain
		\[
		\lim_{x\rightarrow 0}\tfrac{2x}{2^{x}(2^{x}-1)} = \lim_{x\rightarrow 0}\tfrac{2}{2^{2 x} \ln(2)+ 2^x (2^x-1) \ln(2)} = \tfrac{2}{\ln(2)},
		\]
		using de L'H\^{o}pital's rule. Hence, $\Cref{thm: locality} \rightarrow_{K\rightarrow \infty} \tfrac12+ \tfrac{2}{\ln(2)}K(p+1)$ and the statement follows.}
\end{proof}

\begin{theorem}\label{thm: complexity}
	Let $\edges_\mathrm M= \{\edge_0,\dots,\edge_{J-1}\}$ be the set of marked edges used to generate the sequence of T-meshes  $\mesh_0,\mesh_1,\dots,\mesh_J$ with $\mesh_j = (\elements_j,\edges_j,\nodes_j)$ starting from $\mesh_0$, namely  
	\begin{equation*}\mesh_j=\refine(\mesh_{j-1},\edge_{j-1}),\quad \edge_{j-1}\in\edges_{j-1}\quad\text{for}\enspace j\in\{1,\dots,J\}\,.
	\end{equation*} 
	Then, there exists a positive constant $\Cref{thm: complexity}$ such that 
	\begin{equation*}
		\#\edges_J - \#\edges_0 
		\le\ \Cref{thm: complexity} \, \# \edges_\text M.
	\end{equation*}
\end{theorem}
\begin{proof}
	We denote by  $\tilde\edges=\bigcup_{\substack{\mesh\text{ refinement of }\mesh_0}}\edges$ the set of all edges in the initial mesh $\mesh_{{0}}$ and in all of its possible refinements.  For any $\edge'\in \tilde\edges$ and $\edge\in\edges_\text M$, let  
	\begin{equation*}
		\lambda(\edge', \edge)\coloneqq
		\begin{cases}
			2^{\ell(\edge')-\ell( \edge)}&\text{if }\ell(\edge')\le\ell( \edge)+1\text{ and }\dist(\edge', \edge)  <  2^{1-\ell(\edge')}\, \Cref{thm: locality},  
			\\[.3em]
			0&\text{otherwise.}
		\end{cases}
	\end{equation*}
	The proof consists of two main steps devoted to identify
	\begin{itemize}
		\item[(i)] a lower bound for the sum of the $\lambda$ function as $ E$ varies in $\edges_\text M$ so that  each $\edge'\in\edges_J\setminus\edges_0$ satisfies 
		\begin{equation}\label{eq:lb}
			\sum_{ \edge\in\edges_\text M}\lambda(\edge', \edge)\ \ge\ 1,
		\end{equation}
		
		\item[(ii)] an upper bound for the sum of the $\lambda$ function as the refined element $\edge'$ varies in $\edges_J\setminus\edges_0$ so that for any $j = 0,\dots,J-1$ 
		\begin{equation}\label{eq:ub}
			\sum_{\edge'\in\edges_J\setminus\edges_0}\lambda(\edge', \edge_j)
			\ \le \Cref{thm: complexity}
		\end{equation}
		holds.
	\end{itemize}
	If inequalities \eqref{eq:lb} and \eqref{eq:ub} hold for a certain constant $\Cref{thm: complexity}$, we have 
	\begin{equation*}\begin{aligned}[b]
			\#\edges_J - \#\edges_0 
			&\le \#(\edges_J\setminus\edges_0)
			= \sum_{\edge'\in\edges_J\setminus\edges_0} 1
			\\&\le\sum_{\edge'\in\edges_J\setminus\edges_0}\sum_{\edge\in\edges_\text M}\lambda(\edge',\edge)
			\le\sum_{\edge\in\edges_\text M}  \Cref{thm: complexity} 
			=  \Cref{thm: complexity} \, \# \edges_\text M, 
	\end{aligned}\end{equation*}
	and the proof of the theorem is complete. We detail below the proofs of (i) and (ii).
	
	(i)\enspace Let $\edge'\in\edges_J\setminus\edges_0$ be an edge generated in the refinement process from $\mesh_0$ to $\mesh_J$, and let $j_1<J$ be the unique index such that $\edge'\in\edges_{j_1+1}\setminus\edges_{j_1}$.  Theorem~\ref{thm: locality}  states that
	\begin{equation*}
		\dist(\edge',\edge_{j_1})\le 2^{-\ell(\edge')}\, \Cref{thm: locality} 
		\quad \text{and}\quad 
		\ell(\edge')\le\ell(\edge_{j_1})+1\,, 
	\end{equation*}  
	and consequently $\lambda(\edge',\edge_{j_1})=2^{\ell(\edge')-\ell(\edge_{j_1})}>0$.
	The repeated use of Lemma~\ref{thm: locality} yields  a sequence $\{\edge_{j_2},\edge_{j_3},\dots\}$  with {$\edge_{j_{i-1}}\in\edges_{j_i+1}\setminus\edges_{j_i}$},  for $j_1>j_2>j_3>\dots$,  such that 
	\begin{equation}\label{eq: complexity -last}
		\dist(\edge_{j_{i-1}},\edge_{j_i})\le 2^{-\ell(\edge_{j_{i-1}})}\, \Cref{thm: locality} 
		\quad\text{and}\quad 
		\ell(\edge_{j_{i-1}})\le\ell(\edge_{j_i})+1.
	\end{equation}
	We iteratively apply Theorem~\ref{thm: locality} as long as 
	\begin{equation*}
		\lambda(\edge',\edge_{j_i})>0 \quad \text{and} \quad 
		\ell(\edge_{j_i})>0\,,
	\end{equation*} 
	until we reach the first index $L$ with $\lambda(\edge',\edge_{j_L})=0$ or $\ell(\edge_{j_L})=0$.  By considering the three possible cases below,  inequality \eqref{eq:lb} may be derived as follows.
	\begin{itemize}
		\item If $\ell(\edge_{j_L})=0$ and $\lambda(\edge',\edge_{j_L})>0$, then
		\begin{equation*}
			\sum_{ \edge\in\edges_\text M}\lambda(\edge', \edge)\ge\lambda(\edge',\edge_{j_L})=2^{\ell(\edge')-\ell(\edge_{j_L})}  >  1
		\end{equation*}
		since $\ell(\edge') > \ell(\edge_{j_L})  =0$. 
		\item If $\lambda(\edge',\edge_{j_L})=0$ because $\ell(\edge')>\ell(\edge_{j_L})+1$, then \eqref{eq: complexity -last} yields $\ell(\edge_{j_{L-1}})\le\ell(\edge_{j_L})+1<\ell(\edge')$ and hence
		\begin{equation*}\sum_{ \edge\in\edges_\text M}\lambda(\edge', \edge)\ge\lambda(\edge',\edge_{j_{L-1}})=2^{\ell(\edge')-\ell(\edge_{j_{L-1}})}>1.\end{equation*}
		\item If $\lambda(\edge',\edge_{j_L})=0$ because $\dist(\edge',\edge_{j_L})  \ge  2^{1-\ell(\edge')}\, \Cref{thm: locality} $, then a triangle inequality combined with Theorem~\ref{thm: locality} leads to
		\begin{equation*}\begin{aligned}[b]
				2^{1-\ell(\edge')}\, \Cref{thm: locality}  &  \le  
				\dist(\edge',\edge_{j_1})+\sum_{i=1}^{L-1}\dist(\edge_{j_i},\edge_{j_{i+1}}) 
				\le 2^{-\ell(\edge')}\, \Cref{thm: locality} +\sum_{i=1}^{L-1} 2^{-\ell(\edge_{j_i})}\, \Cref{thm: locality}\,.
		\end{aligned}\end{equation*}
		Consequently, $2^{-\ell(\edge')} \le\sum_{i=1}^{L-1} 2^{-\ell(\edge_{j_i})}$ and we obtain
		\begin{equation*} 
			1\ \le\ \sum_{i=1}^{L-1} 2^{\ell(\edge')-\ell(\edge_{j_i})}\ =\ \sum_{i=1}^{L-1} \lambda(\edge',\edge_{j_i})\ \le\ \sum_{ \edge\in\edges_\text M}\lambda(\edge', \edge).\end{equation*}
	\end{itemize}
	
	(ii)\enspace Inequality \eqref{eq:ub} can be derived as follows. For any $0\le j\le J-1$, we consider the set of edges of level $j$ whose distance from $\edge$ is less than $2^{1-j}\Cref{thm: locality}$, defined as 
	\begin{equation*}  A(\edge,j) \sei\bigl\{\edge'\in  \edges(\meshuni j) \mid \dist(\edge', \edge)  < 2^{1-j}\Cref{thm: locality}\}. \end{equation*}
	According to the definition of $\lambda$, the set $A(\edge,j)$ contains all elements at level $j$ that satisfy $\lambda(\edge', \edge)>0$.  We have
	\begin{equation}\label{eq:sumb}
		\sum_{\edge'\in\edges_J\setminus\edges_0}\lambda(\edge', \edge)
		\le \sum_{\edge'\in\bigcup_{j\in\nat}\edges(\meshuni j)\setminus\edges_0}\lambda(\edge', \edge)
		= \sum_{j=1}^{\ell( \edge)+1}2^{j-\ell( \edge)}\,\# 
		A(\edge,j). 
	\end{equation}
	
	If $j<\log_2(2\Cref{thm: locality}/p)$,
	then  $\# A(\edge,j)$ is bounded by $\# A(\edge,j)\le 4^j\# A(\edge,0)$, where $\# A(\edge,0)$ depends only on the initial mesh, in particular on the number of distinct direction indices, and on the number, locations, and valences of extraordinary nodes.
	
	If on the other hand $j\ge\log_2(2\Cref{thm: locality}/p)$,
	we observe that the neighborhood of any edge $\edge\in\edges(\meshuni j)$ contains at most one extraordinary node $v_\edge$. We set $\operatorname{val}(\edge)=\max(4,\text{valence of }v_\edge)$ and observe that 
	$\# A(\edge,j)\le 2\cdot(2\Cref{thm: locality}+1)^2\cdot\operatorname{val}(\edge)$.
	
	Together, we find that
	\begin{equation*}
		\#A(\edge,j) 
		\le \max\bigl\{4^{\log_2(2\Cref{thm: locality}/p)}\# A(\edge,0),2\cdot(2\Cref{thm: locality}+1)^2\cdot\operatorname{val}(\edge)\bigr\}
		\eqqcolon C(\edge).
	\end{equation*}
	
	Finally, the index substitution $k\coloneqq1-j+\ell( \edge)$ reduces \eqref{eq:sumb} to 
	\begin{equation*}\begin{aligned}[b]
			\sum_{\edge'\in\edges_J\setminus\edges_0}\lambda(\edge', \edge) & \le
			\sum_{j=1}^{\ell( \edge)+1}2^{j-\ell( \edge)} \# A(\edge,j) 
			= 
			\sum_{k=0}^{\ell( \edge)}2^{1-k} \# A(\edge,j) \\
			& < 2\sum_{k=0}^\infty 2^{-k} \# A(\edge,j) 
			\le 4\max_{\edge\in\bigcup_{j\in\nat}\edges(\meshuni j)\setminus\edges_0}C(\edge)\,\eqqcolon\Cref{thm: complexity},
	\end{aligned}\end{equation*}
	which completes the proof.
\end{proof}
{With Theorems~\ref{thm: locality} and~\ref{thm: complexity}, we have that the refinement obtained by Algorithm~\ref{alg: ref} remains local and exhibits linear complexity.}

\subsection{Linear independence}

In this subsection, we show that the T-spline basis functions defined over a refined T-mesh are linearly independent.
\begin{theorem}[quasi-uniformity of generated meshes]\label{thm: quasi-uniformity}%
	For any T-mesh $\mesh=(\nodes,\edges,\elements)$ generated by $\refine$ (see Algorithm~\ref{alg: ref}), the edge levels in any edge neighborhood are bounded from above and below by
	\begin{equation}\label{eq: quasi-uniformity}
		\Forall \edge\in\edges,\ \edge' \in \neighb(\mesh,\edge):\left\{
		\begin{alignedat}{2}
			\ell(\edge) &\le\ell(\edge')\le\ell(\edge)+1 &\quad\text{if } \di(\edge')&<\di(\edge) 
			\\ \ell(\edge)-1 &\le\ell(\edge')\le\ell(\edge)+1 &\quad\text{if } \di(\edge')&=\di(\edge) 
			\\ \ell(\edge)-1 &\le\ell(\edge')\le\ell(\edge)   &\quad\text{if } \di(\edge')&>\di(\edge).
		\end{alignedat}\right.
	\end{equation}
\end{theorem}
\begin{proof}
	We prove the theorem by induction over admissible subdivisions.
	We call  $\check\mesh=\subdiv(\mesh,\edge)$ an \emph{admissible subdivision} if 
	\begin{equation}\label{eq: cond admissible subdiv}
		\Forall \edge'\in\neighb(\mesh,\edge):
		\begin{cases}
			\ell(\edge')>\ell(\edge) &\text{if }\di(\edge')<\di(\edge)
			\\ \ell(\edge')\ge\ell(\edge) &\text{if }\di(\edge')\ge\di(\edge).
		\end{cases}
	\end{equation}
	Note that $\refine$ always performs a sequence of admissible subdivisions. Our approach is therefore appropriate for proving the theorem for all T-meshes.
	
	The claim is true for the initial mesh $\mesh_0$ since all edges $\edge,\edge'\in\edges_0$ satisfy $\ell(\edge)=\ell(\edge')=0$.
	For the induction proof, we consider a T-mesh $\mesh$ that satisfies \eqref{eq: quasi-uniformity} and $\edge\in\edges$ satisfying \eqref{eq: cond admissible subdiv}. Together, this yields
	\begin{equation}\label{eq:thm:qu: level bounds}
		\Forall \edge' \in \neighb(\mesh,\edge):\left\{
		\begin{alignedat}{2}
			\ell(\edge')&=\ell(\edge)+1 &\quad\text{if } \di(\edge')&<\di(\edge) 
			\\ \ell(\edge)\, &\le\ell(\edge')\le\ell(\edge)+1 &\quad\text{if } \di(\edge')&=\di(\edge) 
			\\ \ell(\edge')&=\ell(\edge)   &\quad\text{if } \di(\edge')&>\di(\edge).
		\end{alignedat}\right.
	\end{equation}
	We aim to show that $\check\mesh=\subdiv(\mesh,\edge)$ also fulfills the claim.
	The subdivision of $\edge$ removes the edge $\edge$ and adds two children $\edge_1$, $\edge_2$ to the mesh, and maybe $\tilde \edge$ or $\tilde \edge_1,\tilde \edge_2$ if one or two neighboring quadrilaterals are subdivided.
	We will below show the claim for
	\begin{enumerate}[label=\textsf{(\roman*)}]
		\item $\edge_1,\edge_2$, \label{enum:thm:qu: children}
		\item $\edge'\in\check\edges\cap\edges$ with $\edge\in\neighb(\mesh,\edge')$, \label{enum:thm:qu: oldEinNE}
		\item $\tilde \edge_1,\tilde \edge_2$, \label{enum:thm:qu: tildeE}
		and
		\item $\edge'\in\check\edges\cap\edges$ with $\edge\notin\neighb(\mesh,\edge')$. \label{enum:thm:qu: oldEnotinNE}
	\end{enumerate}
	\newcommand{\enumref}[1]{\noindent\ref{#1}\enspace}
	
	\enumref{enum:thm:qu: children} Any edge $\edge'\in\neighb(\check\mesh,\edge_1)$ satisfies
	\[ \dist(\edge',\edge)\le \dist(\edge',\edge_1) + \dist(\edge_1,\edge) 
	\le \tfrac{p+1}2\cdot2^{-\ell(\edge_1)} + 2^{-1-\ell(\edge_1)}
	\le  \tfrac{p+1}2\cdot2^{-\ell(\edge)}
	\]
	since $\tfrac{p+1}2\ge \frac12$. Hence all edges $\edge'\in\neighb(\check\mesh,\edge_1)\cap\edges$ are in $\neighb(\mesh,\edge)$ and satisfy \eqref{eq:thm:qu: level bounds}. With $\ell(\edge_1)=\ell(\edge)+1$ and $\di(\edge_1) = \di(\edge)$,
	we can rewrite this as
	\begin{equation}\label{eq:thm:qu: level bounds for children}
		\Forall \edge' \in \neighb(\check\mesh,\edge_1)\cap\edges:\left\{
		\begin{alignedat}{2}
			\ell(\edge')&=\ell(\edge_1) &\quad\text{if } \di(\edge')&<\di(\edge_1) 
			\\ \ell(\edge_1)-1 &\le\ell(\edge')\le\ell(\edge_1) &\quad\text{if } \di(\edge')&=\di(\edge_1) 
			\\ \ell(\edge')&=\ell(\edge_1)-1   &\quad\text{if } \di(\edge')&>\di(\edge_1).
		\end{alignedat}\right.
	\end{equation}
	For the remaining edges $\edge'\in \neighb(\check\mesh,\edge_1)\setminus\edges$, which are $\edge_1,\edge_2,\tilde \edge_1,\tilde \edge_2$ from above, we know that $\ell(\edge_2)=\ell(\edge_1)=\ell(\edge)+1$ and that $\tilde \edge_1,\tilde \edge_2$ have the same level and direction index as their opposite edges (see Definitions \ref{df: direction index} and \ref{df: refinement-level}), at least one of which is in $\neighb(\check\mesh,\edge_1)\cap\edges$ and satisfies \eqref{eq:thm:qu: level bounds for children}.
	Hence all new edges also satisfy \eqref{eq:thm:qu: level bounds for children} and the claim is fulfilled by $\edge_1$, and similarly by $\edge_2$.

	\enumref{enum:thm:qu: oldEinNE} Let $\edge'\in\check\edges\cap\edges$ with $\edge\in\neighb(\mesh,\edge')$. The induction hypothesis yields that 
	\begin{equation}\label{eq:thm:qu: level bounds tildeE}
		\Forall  \edge'' \in \neighb(\mesh,\edge'):\left\{
		\begin{alignedat}{2}
			\ell(\edge') &\le\ell(\edge'')\le\ell(\edge')+1 &\quad\text{if } \di(\edge'')&<\di(\edge') 
			\\ \ell(\edge')-1 &\le\ell(\edge'')\le\ell(\edge')+1 &\quad\text{if } \di(\edge'')&=\di(\edge') 
			\\ \ell(\edge')-1 &\le\ell(\edge'')\le\ell(\edge')   &\quad\text{if } \di(\edge'')&>\di(\edge').
		\end{alignedat}\right.
	\end{equation}
	Since $\edge'$ is unchanged, all old neighbor edges $\edge'' \in \neighb(\check\mesh,\edge')\cap\neighb(\mesh,\edge')$ satisfy \eqref{eq:thm:qu: level bounds tildeE} also in the new mesh $\check\mesh$.
	The set of new neighbor edges $\neighb(\check\mesh,\edge')\setminus\neighb(\mesh,\edge')$ 
	may contain at most $\edge_1,\edge_2,\tilde \edge_1$ and $\tilde \edge_2$, and it contains at least one of the children $\edge_1,\edge_2$.
	Since \eqref{eq:thm:qu: level bounds tildeE} also holds for $\edge''=\edge$ with $\ell(\edge_1)=\ell(\edge)+1$ and $\di(\edge_1)=\di(\edge)$, we have
	\[\begin{alignedat}{2}
		\ell(\edge')+1 &\le\ell(\edge_1)\le\ell(\edge')+2 &\quad\text{if } \di(\edge_1)&<\di(\edge') 
		\\ \ell(\edge') &\le\ell(\edge_1)\le\ell(\edge')+2 &\quad\text{if } \di(\edge_1)&=\di(\edge') 
		\\ \ell(\edge') &\le\ell(\edge_1)\le\ell(\edge')+1   &\quad\text{if } \di(\edge_1)&>\di(\edge').
	\end{alignedat}
	\]
	Assume for contradiction that $\di(\edge_1)\le\di(\edge')$ and $\ell(\edge_1)=\ell(\edge')+2$ or that 
	$\ell(\edge_1)=\ell(\edge')+1$ and $\di(\edge_1)>\di(\edge')$.
	We conclude that $\ell(\edge)=\ell(\edge')+1$ or $\ell(\edge)=\ell(\edge')$, respectively, and together with $\edge\in\neighb(\mesh,\edge')$ that $\dist(\edge,\edge')\le \tfrac{p+1}2\cdot 2^{-\ell(\edge')}\le \tfrac{p+1}2\cdot 2^{-\ell(\edge)}$ and hence $\edge'\in\neighb(\mesh,\edge)$. Thus, $\edge'$ satisfies \eqref{eq:thm:qu: level bounds} which leads to contradiction in both cases.
	This shows the claim for $\edge_1 \in \neighb(\check\mesh,\edge')$, and similarly for $\edge_2$.
	
	Consider $\tilde \edge\in\{\tilde \edge_1,\tilde \edge_2\}$, which has two opposite edges with the same refinement level and direction index. At least one of these opposite edges is in $\neighb(\check\mesh,\edge')\cap\neighb(\mesh,\edge')$ and satisfies \eqref{eq:thm:qu: level bounds tildeE}, and so does $\tilde \edge$.
	
	\enumref{enum:thm:qu: tildeE} 
	Consider again $\tilde \edge\in\{\tilde \edge_1,\tilde \edge_2\}$, which has two opposite edges $\tilde \edge_1,\tilde \edge_2$ with the same refinement level and direction index. This and the construction of the metric $\dist$ yield that 
	\[\neighb(\check\mesh,\tilde \edge)\subseteq \neighb(\check\mesh,\tilde \edge_1)\cup \neighb(\check\mesh,\tilde \edge_2).\]
	The claim has been shown for $\tilde \edge_1,\tilde \edge_2$ in \ref{enum:thm:qu: oldEinNE}, and it also holds for $\tilde \edge$ due to the equality of levels and direction indices.
	
	\enumref{enum:thm:qu: oldEnotinNE} We consider $\edge'\in\check\edges\cap\edges$ with $\edge\notin\neighb(\mesh,\edge')$. In order to prove the claim, we need to show that any new edges $\edge''\in\neighb(\check\mesh,\edge')\setminus\neighb(\mesh,\edge')$ satisfy \eqref{eq: quasi-uniformity}. We therefore distinguish three cases.
	
	\emph{Case 1:} $\tilde \edge\in\neighb(\check\mesh,\edge')\setminus\neighb(\mesh,\edge')$ from the subdivision of a neighbor element of $\edge$. The edge $\tilde \edge$ has two opposite edges $\tilde \edge_1,\tilde \edge_2$ with the same refinement level and direction index.
	Due to the construction of the metric $\dist$, at least one edge $\tilde \edge\in\{\tilde \edge_1,\tilde \edge_2\}$ is in $\neighb(\check\mesh,\edge')\cap\neighb(\mesh,\edge)$ and satisfies \eqref{eq:thm:qu: level bounds tildeE}, and $\tilde \edge$ satisfies the same due to the equality of levels and direction indices.
	
	\emph{Case 2:} $\edge_1\in\neighb(\check\mesh,\edge')\setminus\neighb(\mesh,\edge')$, with $\edge_1$ being a child of $\edge$ and $\ell(\edge')\ge\ell(\edge)+1$.
	In this case we get
	\begin{align*}
		\dist(\edge',\edge)&\le \dist(\edge',\edge_1) + \dist(\edge_1,\edge) 
		\le \tfrac{p+1}2\cdot2^{-\ell(\edge')} + 2^{-2-\ell(\edge)}
		\\&\le \bigl(\tfrac{p+1}2+\tfrac12\bigr)\cdot2^{-1-\ell(\edge)}
		\le \tfrac{p+1}2\cdot2^{-\ell(\edge)}
		\\\Rightarrow\quad \edge'&\in\neighb(\mesh,\edge)
		\quad\stackrel{\eqref{eq:thm:qu: level bounds}}\Rightarrow\quad \ell(\edge')= \ell(\edge)+1=\ell(\edge_1).
	\end{align*}
	
	\emph{Case 3:} $\edge_1\in\neighb(\check\mesh,\edge')\setminus\neighb(\mesh,\edge')$, with $\edge_1$ being a child of $\edge$ and $\ell(\edge')\le\ell(\edge)$.
	The edge midpoints of both $\edge'$ and $\edge$ are nodes in the uniform mesh $\meshuni L$ of level $L=\max(\ell(\edge'),\ell(\edge))+1 = \ell(\edge)+1=\ell(\edge_1)$. By definition of the mesh metric $\dist$, we have $\dist(\edge',\edge)=\dist(\midp(\edge'),\midp(\edge))=n\cdot2^{-\ell(\edge_1)}$ with $n\in\nat$ being the minimal number of elements connecting $\midp(\edge')$ and $\midp(\edge)$ in the mesh $\meshuni L$.
	A triangle inequality and $\edge_1\in\neighb(\check\mesh,\edge')$ yield
	\begin{equation*}
		\dist(\edge',\edge)\le \dist(\edge',\edge_1) + \dist(\edge_1,\edge) 
		\le \tfrac{p+1}2\cdot2^{-\ell(\edge')} + 2^{-1-\ell(\edge_1)}
		= \bigl(\tfrac{p+1}2 \cdot2^{\ell(\edge_1)-\ell(\edge')} + \tfrac12 \bigr)\cdot2^{-\ell(\edge_1)}.
	\end{equation*}
	Hence, we expect $n$ to be bounded by $n\le \tfrac{p+1}2 \cdot2^{\ell(\edge_1)-\ell(\edge')} + \tfrac12$. Since both $n$ and $\tfrac{p+1}2 \cdot2^{\ell(\edge_1)-\ell(\edge')}$ are integer numbers, we conclude that $n\le \tfrac{p+1}2 \cdot2^{\ell(\edge_1)-\ell(\edge')}$, hence
	$\dist(\edge',\edge)\le \tfrac{p+1}2\cdot2^{-\ell(\edge')}$ and finally $\edge\in\neighb(\mesh,\edge')$ in contradiction to the assumption above.
\end{proof}

\begin{df}[extension and skeleton of T- and I-nodes]
	Given a T-node or boundary I-node $\node\in\nodes$ (see Definition~\ref{df: T-node and I-node}), 
	there is a unique element $\element\in\elements$ with two common edges of equal direction index, $\edges(\node)\cap\edges(\element)=\{\edge_1,\edge_2\}$, $\edge_1\ne\edge_2$, $\di(\edge_1)=\di(\edge_2)$,
	that have a common opposite edge $\edge_\text{o}\in\edges(\element)$.
	The \emph{first-order extension} $\ext^1(\node)$ is a singleton set containing the edge that connects $\node$ with the midpoint of $\edge_\text{o}$.
	For an interior I-node $\node$, there are two such neighboring elements $\element_1,\element_2$, and $\ext^1(\node)$ is a set of two edges related to these two elements.
	In both cases, we have $\ext^1(\node)\subset\bedges\setminus\edges$.
	
	The \emph{T-node extension} (also for I-nodes) is the $\frac{p-1}2$-th prolongation of $\ext^1(\node)$, 
	\[ \ext(\node) = \ep^{(p-1)/2}(\ext^1(\node))\subset\bedges.\] 
	
	For T-nodes and boundary I-nodes,
	we define the \emph{first-order skeleton} of $\node$ as $\sk(\ext^1(\node))=\{\edge_1,\edge_2,\edge_\text o\}$, and 
	for interior I-nodes, $\sk(\ext^1(\node))=\{\edge_1,\edge_2,\edge_{\text o,1},\edge_{\text o,2}\}$,  with $\edge_{\text o,1}\in\edges(\element_1)$ and $\edge_{\text o,2}\in\edges(\element_2)$ the common opposite edges of $\edge_1$ and $\edge_2$.
	The \emph{$n$-th-order skeleton} is the union of the $(n-1)$-th-order skeleton and their opposite edges, for $n=1,\dots,\tfrac{p+1}2$.
	The \emph{extension skeleton} $\sk(\ext(\node))$ of $\node$ is defined as its $\tfrac{p+1}2$-th-order skeleton.
	
	Note that all edges in $\sk(\ext(\node))$ have the same direction index due to the criterion \eqref{eq: di criterion}. Hence, we can associate to each T-node or I-node the unique direction index of its extension skeleton,
	and we call $\node$ with $\di(\sk(\ext(\node)))=\{j\}$ a \emph{$j$-orthogonal} T-node (or I-node).
\end{df}

\begin{theorem}\label{thm: analysis-suitable}
	For any T- or I-nodes {$\node,\node'$} in a mesh $\mesh$ generated by $\refine$, where {$\node$ is $i$-orthogonal, $\node'$ is $j$-orthogonal, and $i\ne j$}, it holds $\ext({\node})\cap\ext({\node'})=\emptyset$, i.e., the T-node extensions do not intersect.
\end{theorem}
\begin{proof}
	Analogously to the definition above, there is $\element\in\elements({\node'})$  with common edges $\edge_1,\edge_2\in\edges({\node'})\cap\edges(\element)$ with $\di(\edge_1)=\di(\edge_2)=j$, and an opposite edge $\edge_\text o\in\edges(\element)$ with $\di(\edge_\text o)=j$.
	Since the mesh $\mesh$ is generated by $\refine$, the subdivision of $\edge_1$'s parent edge $\edge_\text p$ was admissible in the sense of \eqref{eq: cond admissible subdiv} and we conclude, since the mesh may have been refined further,  that 
	\[ \Forall \edge'\in\neighb(\mesh,\edge_\text p):
	\begin{cases}
		\ell(\edge')>\ell(\edge_\text p) &\text{if }\di(\edge')<j
		\\ \ell(\edge')\ge\ell(\edge_\text p) &\text{if }\di(\edge')\ge j.
	\end{cases}
	\]
	
	For the opposite edge $\edge_\text o\in\edges$, which is not refined, Theorem~\ref{thm: quasi-uniformity} yields
	\[  \Forall \edge' \in \neighb(\mesh,\edge_\text o):\left\{
	\begin{alignedat}{2}
		\ell(\edge_\text o) &\le\ell(\edge')\le\ell(\edge_\text o)+1 &\quad\text{if } \di(\edge')&<j
		\\ \ell(\edge_\text o)-1 &\le\ell(\edge')\le\ell(\edge_\text o)+1 &\quad\text{if } \di(\edge')&=j
		\\ \ell(\edge_\text o)-1 &\le\ell(\edge')\le\ell(\edge_\text o)   &\quad\text{if } \di(\edge')&>j.
	\end{alignedat}\right.
	\]
	
	Together with $\ell(\edge_\text p)=\ell(\edge_\text o)$, we have
	\[  \Forall \edge' \in \neighb(\mesh,\edge_\text o)\cap\neighb(\mesh,\edge_\text p):\left\{
	\begin{alignedat}{2}
		\ell(\edge')&=\ell(\edge_\text o)+1 &\quad\text{if } \di(\edge')&<j
		\\ \ell(\edge_\text o) &\le\ell(\edge')\le\ell(\edge_\text o)+1 &\quad\text{if } \di(\edge')&=j
		\\ \ell(\edge')&=\ell(\edge_\text o)   &\quad\text{if } \di(\edge')&>j.
	\end{alignedat}\right.
	\]
	The size of the joint neighborhood $ \neighb(\mesh,\edge_\text o)\cap\neighb(\mesh,\edge_\text p) $ is hence $p$ elements in $j$-orthogonal direction, centered at $\element$, times $p+2$ elements in the other direction(s), centered at $\element$.
	Elements in this neighborhood may have been subdivided in $j$-th direction, but in no other direction. Hence any $i$-orthogonal T-node ${\node}$, $i\ne j$, is far away in the sense that its extension cannot intersect with the extension of ${\node'}$.
\end{proof}

\begin{rem}\label{rem: tj-exts far away}
	The theorem above implies that for each iteration in Algorithm~\ref{alg: bezier mesh} for the B\'ezier mesh, the computed mesh contains only elements $\element\in\elements$ with either $\#\nodes(\element)=\#\edges(\element)=4$ or $\#\nodes(\element)=\#\edges(\element)=5$, where the latter are exactly those elements that neighbor a T- or I-node. 
\end{rem}
Theorem~\ref{thm: analysis-suitable} states that the T-splines defined over the T-mesh $\mesh$ are analysis-suitable. This implies that the functions are linearly independent. Moreover, on all elements $\element \in \elements \setminus \bigcup_{\node \in \enodes} \dc_p(\node)$ the functions are locally linearly independent.
\begin{theorem}[linear independence]\label{thm: linear-idependence}
	If the mesh $\mesh$ is generated by $\refine$, then the functions in $\cB$ as given in Definition~\ref{df: T-splines over an unstructured T-mesh} are linearly independent.
\end{theorem}
\begin{proof}
	The proof consists of three steps:
	\begin{itemize}
		\item[a)] For each extraordinary node $\node\in\enodes$, the functions in $\cB_\node$, together with the anchors $\anchors[\node] = \{\node'\in\anchors: \dist(\node, \node') = (p+1)/2\}$, span all functions in the $1$-disk $\dc_1(\node)$, i.e.,
		\[
		\mathrm{span} (\cB_\node \cup \anchors[\node])|_{\dc_1(\node)} = \splines^p(\bmesh,\kspecial)|_{\dc_1(\node)}.
		\]
		Moreover, they are linearly independent by construction.
		\item[b)] For all other anchors $\node' \in\anchors$, that are not contained in $\anchors[\node]$ for any extraordinary node $\node$, we have
		\[
		B_{\node'} |_{\dc_1(\node)} \equiv 0
		\]
		for all $\node \in \enodes$. Thus, they are linearly independent to the functions in (a).
		\item[c)] For all functions $B_{\node}$ with $\node\in\anchors$, one can define functionals $\lambda_{\node}$ on the support of $B_{\node}$, which form a dual basis, see, e.g.,~\cite{STV:2016}. Since for any pair of anchors $\node,\node'\in\anchors$ the supports of $B_{\node}$ and $\lambda_{\node'}$ are contained inside a structured submesh, due to Assumption~\ref{assu:disk-around-anchors}, the analysis-suitability condition in Theorem~\ref{thm: analysis-suitable} yields dual compatibility, i.e., $\lambda_{\node'} (B_{\node}) = \delta_{\node,\node'}$.
	\end{itemize}
	Consequently, all functions are linearly independent. 
\end{proof}
A similar statement is given for specific $C^1$/$C^2$-smooth cubic T-splines in~\cite[Proposition 3.2]{TosSH17}, where the proof follows a similar reasoning. See also~\cite{SSELBHS:2013,Li2015} for related approaches.

\section{Conclusions \& Outlook}\label{sec: conclusions}

We have introduced an adaptive refinement algorithm for T-splines on unstructured meshes in two dimensions. Motivated by ideas for higher-dimensional structured T-meshes, the refinement routine is based on the concept of so-called direction indices {that are involved in a recursive refinement procedure}. We have proved that the refinement algorithm has linear complexity and preserves analysis-suitability away from extraordinary nodes.
Additionally, we sketched a generalization to more general unstructured meshes and drew the connections to manifold splines and Isogeometric Analysis.

For the preservation of global analysis-suitability, we require uniform refinement in the $p$-disk of extraordinary nodes at the current stage, while the treatment of T-nodes in these regions is left to future work. Moreover, a combination of our adaptive refinement scheme with special constructions near extraordinary nodes, based on special splits and/or the imposition of smoothness of higher order as, e.g., in~\cite{SSELBHS:2013,CLBZG:2016,CWTLHKZ:2020,TosSH17,WLQHZC:2021}, is of great relevance for practical applications.  
Further, a natural extension of our work is the application to physical problems in the setting of Isogeometric Analysis, along with numerical experiments investigating the condition numbers and sparsity patterns of the system matrices.

\section*{Acknowledgments}

Roland Maier acknowledges support by the German Research Foundation (DFG) in the Priority Program 1748 \emph{Reliable simulation techniques in solid mechanics} (PE2143/2-2) and by the G\"oran Gustafsson Foundation for Research in Natural Sciences and Medicine. 
The research of Thomas Takacs is partially supported by the Austrian Science Fund (FWF) and the government of Upper Austria through the project P~30926-NBL entitled ``Weak and approximate $C^1$-smoothness in isogeometric analysis''.

\appendix

\section{Manifold interpretation of unstructured meshes}\label{sec: manifold interpretation}

In this section, we generalize the setting introduced in Section~\ref{sec: Preliminaries} from planar partitions to manifolds. The notion of an unstructured T-mesh extends quite directly to surfaces and manifold-like domains. In Section~\ref{sec: Partitions-Meshes}, we defined meshes via partitions of planar, polygonal domains. However, the functions on such unstructured T-meshes are defined locally, either on structured submeshes or on specific submeshes around extraordinary nodes. The entire polygonal domain is thus covered by overlapping submeshes and more precisely, every point of the partition, except for extraordinary nodes, is in the interior of some structured submesh. All structured submeshes $\mesh'$ possess a corresponding parameter domain $\domain{\pullback{\mesh'}}$, as introduced in Definition~\ref{df: structured-mesh}. In that definition, $\pullback{\mesh'}$ is a mesh composed entirely of axis-aligned rectangles. Any function $\varphi\in\cB$ defined on the mesh $\mesh$ has a spline representation on the parameter domain of each submesh, i.e., for the submesh $\mesh'$ we have the spline
\[
\pullback{\varphi}' : \domain{\pullback{\mesh'}} \rightarrow \real.
\]
The spline representations on different submeshes $\mesh'$ and $\mesh''$ are related through the mappings $\mapping{\elements'}$ and $\mapping{\elements''}$ introduced in Definition~\ref{df: structured-mesh}, via $\pullback{\varphi}' \circ \mapping{\elements'}^{-1} \circ \mapping{\elements''} = \pullback{\varphi}''$, which is defined only on the parameter domain of the intersection $\domain{\mesh'} \cap \domain{\mesh''}$.

This underlying structure suggests that the submeshes and parameter domains, together with the mappings $\mapping{\elements'}^{-1} \circ \mapping{\elements''}$, define a manifold. Splines can then be defined on this manifold, following the approach in~\cite{STV:2016}. The following subsections are devoted to defining T-meshes and T-splines based on such a manifold interpretation. 

\subsection{Meshes on parameter manifolds}

We follow~\cite{STV:2016}, which is based on~\cite{GH:1995}, to introduce an abstract representation of the parameter domain, the so-called \emph{parameter manifold}. 
For simplicity, we provide definitions only for two-dimensional manifolds. Extensions to general dimensions $d\geq 2$ can be found in~\cite{STV:2016}.

\begin{df}[proto-manifold]\label{df: proto-manifold} A proto-manifold of dimension two
	consists of
	\begin{itemize}
		\item a finite set $\{\omega_i\}_{i=1,\ldots,N}$  (named \emph{proto-atlas}) of \emph{charts} $\omega_i$, that are open polytopes $ \omega_i \subset \real^2$;
		\item a set of open \emph{transition domains}  $\{
		\omega_{i,j}\}_{i,j=1,\ldots,N}$ that are polytopes such that  $  \omega_{i,j} \subset   \omega_{i} $, $
		\omega_{i,i} =  \omega_{i} $,  and each  $  \omega_{i,j} $ is the
		interior of its closure;
		\item a set of \emph{transition functions} $\{  \psi_{i,j}\}_{i,j=1,\ldots,N}$, that are
		homeomorphisms  $\psi_{i,j} : \omega_{i,j}  \rightarrow  \omega_{j,i} $ fulfilling 
		the \emph{cocycle condition} $\psi_{j,k}\circ \psi_{i,j} =\psi_{i,k} $ in $ \omega_{i,j}  \cap  \omega_{i,k}  $ for all $i,j,k=1,\ldots,N$, where $\psi_{ii}=\operatorname{id}$;
		\item For every $i,j$, with $i\neq j$, for every $\boldsymbol{\zeta}_i \in \partial \omega_{i,j} \cap \omega_{i}$ and $\boldsymbol{\zeta}_j \in \partial \omega_{j,i} \cap \omega_{j}$, there are open balls, $V_{\boldsymbol{\zeta}_i}$ and $V_{\boldsymbol{\zeta}_j}$, centered at $\boldsymbol{\zeta}_i$ and $\boldsymbol{\zeta}_j$, such that no point of $V_{\boldsymbol{\zeta}_j} \cap \omega_{j,i}$ is the image of any point of $V_{\boldsymbol{\zeta}_i}\cap\omega_{i,j}$ by $\psi_{i,j}$.
	\end{itemize}
\end{df}
Note that the last condition is taken from~\cite{SXGNMV:2009}. As pointed out in~\cite{STV:2016}, it prevents bifurcations of the domain and guarantees that the resulting object is indeed a manifold.

It is easy to see that on a planar partition of $\Omega$, which is segmented into overlapping, open subdomains such that $\Omega = \bigcup_{i} \Omega_i$, each $\Omega_i$ is an open polytope. Thus, the choice $\omega_i = \Omega_i$ as charts, $\omega_{i,j} = \Omega_{i,j} = \Omega_i \cap \Omega_j$ as transition domains with transition functions $\psi_{i,j}$ being the identity mapping on $\omega_{i,j}$ forms a proto-manifold. However, such a structure complicates the definition of smooth spline functions over the mesh.
Alternatively, we can define for every structured submesh $\mesh_i$ over the domain $\Omega_i = \domain{\mesh_i}$ the corresponding chart as $\omega_i=\domain{\pullback{\mesh_i}}$. In addition to such structured charts, one can introduce for each extraordinary node $\node \in \enodes$ a specific unstructured chart, which corresponds to the domain $\domain{\dc_{p}(\node)}$ of the $p$-disk around the node.

By merging and identifying the charts of the  proto-manifold, we obtain a manifold, which can serve as the parameter domain, thus the name \emph{parameter manifold}.
\begin{df}[parameter manifold] \label{df: parameter-manifold}
	Given a proto-manifold, the set 
	\begin{equation}\label{eq:disj-union-of-Omega-i}
		\Omega^P = \left ( \bigsqcup_{i=1,\ldots,N}  \omega_i\right )
		\bigg/\sim
	\end{equation}
	is called a \emph{parameter manifold}. 
	Here $\bigsqcup$ denotes the disjoint union, i.e., 
	\[\bigsqcup_{i=1,\ldots,N}  \omega_i = \left \{[\bzeta_i,i] , \,
	\bzeta_i\in \omega_i,i=1,\ldots,N \right \}
	\]
	and the equivalence relation $\sim$ is defined for all $\bzeta_i\in
	\omega_i$ and $\bzeta_j\in  \omega_j$, as
	\[
	[\bzeta_i,i] \sim
	[\bzeta_j,j] \Leftrightarrow \psi_{i,j}(\bzeta_i)=\bzeta_j.
	\]
	We denote by $\pi_{i}(\bzeta_i) \in \Omega^P$ the equivalence class corresponding to $\bzeta_i \in  \omega_i$. 
\end{df}
In the following, we use the notation  $ \Omega^P_i = \pi_{i}( \omega_i)$ and  $ \Omega^P_{i,j} = \pi_{i}( \omega_{i,j})$, where  $  \Omega^P_i,\Omega^P_{i,j} \subset \Omega^P$. Hence, we have $ \Omega^P_{i,j} = \Omega^P_{j,i} =  \Omega^P_{i} \cap \Omega^P_{j}$.  Note that $\pi_{i}$ is a one-to-one correspondence
between each $ \omega_i$ and  $\Omega^P_i $ and plays the role of a local representation. 
Next, we define a proto-mesh on the charts. 
\begin{df}[proto-mesh]\label{df: proto-mesh}
	A \emph{proto-mesh} on a proto-manifold is a collection of conforming meshes, i.e., it is a set 
	\begin{equation*}
		\label{eq:mesh-i}
		\{\tau_i \}_{i=1,\ldots,N} \hspace{10pt} \mbox{ with } \tau_i= \{ \elementSmall \subset  \omega_i \},
	\end{equation*}
	where
	\begin{enumerate}[label=(P\arabic*)]
		\item each set $\tau_i$ is composed of axis-parallel, open
		rectangles $\elementSmall$, called \emph{elements}, and each set $\tau_i$ is a mesh on $\omega_i$, i.e. the elements are disjoint 
		and the union of the closures of the elements is the closure of $\omega_i $,
		\item for every $i,j$, $\omega_{i,j}$ is the interior of the 
		union of the closure of elements of $\tau_i$, and 
		\item \label{pm: trans fncts map elems to elems} the transition functions $\psi_{i,j}$ map elements onto elements, i.e.,
		\begin{equation*}
			\label{eq:mesh-compatile-with-transitions}
			\forall \elementSmall \in \tau_i, \,  \psi_{i,j} (\elementSmall) \in \tau_j.
		\end{equation*}
		Similarly, the transition functions map vertices of $\elementSmall$ onto vertices of $\psi_{i,j} (\elementSmall)$ and edges of $\elementSmall$ onto edges of $\psi_{i,j} (\elementSmall)$.
	\end{enumerate}
\end{df}
Note that in addition to mesh elements one can also define the edges and nodes of the mesh derived from the elements in $\tau_i$. To simplify the notation, we omit those definitions.

The proto-mesh naturally defines a mesh on $\Omega^P$ through the mappings $\pi_i$.
\begin{df}[mesh on $\Omega^P$] We define the mesh on the
	parameter manifold  $\Omega^P$ as $\mesh = (\elements,\edges,\nodes)$, with 
	\begin{equation*}
		\label{eq:mesh}
		\elements = \{ \element \subset \Omega^P \mid \element = \pi_{i}(\elementSmall), \elementSmall \in    \tau_i\text{ for some } i=1,\ldots,N \}.
	\end{equation*}
	The property \ref{pm: trans fncts map elems to elems} allows analogous definitions of the edges 
	\[\edges = \{\edge\subset \Omega^P\mid \edge = \pi_{i}(\edgeSmall),\element = \pi_{i}(\elementSmall), \edgeSmall\in\edges(\elementSmall),\elementSmall\in\tau_i\text{ for some }i=1,\ldots,N\}\]
	\\[-1.75em]and nodes\\[-1.25em]
	\[\nodes = \{\node\subset \Omega^P\mid \node = \pi_{i}(\nodeSmall),\element = \pi_{i}(\elementSmall), \nodeSmall\in\nodes(\elementSmall),\elementSmall\in\tau_i\text{ for some }i=1,\ldots,N\}\]
	of the mesh $\mesh$. 
\end{df}
Each chart $\omega_i$ defines a submesh $\mesh_i$ with elements $\elements_i = \{\pi_i(\elementSmall),\elementSmall\in\tau_i\}$. 
Given this alternative definition of a mesh and submesh, we can also extend the definition of a structured (sub)mesh to manifolds. A chart $\omega_i$ and local mesh $\tau_i$ are called \emph{structured} if $\tau_i$ is a structured mesh according to Definition~\ref{df: structured-mesh}. The definition can be extended to larger submeshes composed of several charts and is directly applicable to T-meshes as in Section~\ref{subsec: T-meshes}.

If we are given a partition of a polygonal domain $\Omega$, then $\Omega^P_i = \Omega_i$ for all subdomains and we can associate $\Omega^P$ with $\Omega$. Each element of the equivalence class $\Omega^P$ corresponds to a point in $\Omega$. For general parameter manifolds $\Omega^P$, there exists a domain $\Omega \subset \real^n$ and a bijective embedding $\mathbf{G}:\Omega^P \rightarrow \Omega$. This yields a mapping $\mapping{\elements_i} = \mathbf{G} \circ \pi_{i} : \omega_i \rightarrow \Omega_i \subseteq \Omega$. We can thus define T-splines over the manifold domain $\Omega$.

\subsection{Splines over meshes on parameter manifolds}

To define splines over T-meshes on a parameter manifold, in analogy to Section~\ref{subsec: T-meshes}, the functions need to be defined consistently as piecewise polynomials and the present setting requires a clear definition of continuity over edges.

Each spline function $\varphi:\Omega\rightarrow\real$ is defined as a piecewise polynomial over structured charts. Thus, for each structured chart $\omega_i$ the functions
\[
\pullback{\varphi}_i : \omega_i \rightarrow \real,
\]
with $\pullback{\varphi}_i = \varphi \circ \mapping{\elements_i}$, are polynomials on each element in $\tau_i$. Since the functions need to be consistent for all elements shared by several structured charts, we need for all structured charts $\omega_i$ and $\omega_j$ that 
\[
\pullback{\varphi}_i \circ {\mapping{\elements_i}}^{-1} \circ {\mapping{\elements_j}} = \pullback{\varphi}_i \circ {\pi_i}^{-1} \circ {\mathbf{G}}^{-1} \circ \mathbf{G} \circ {\pi_j} = \pullback{\varphi}_i \circ {\pi_i}^{-1} \circ {\pi_j} = \pullback{\varphi}_i \circ \psi_{j,i} = \pullback{\varphi}_j.
\]
For both functions $\pullback{\varphi}_i$ and $\pullback{\varphi}_j$ to be polynomials of a fixed bi-degree $(p,p)$ on each element, we need the transition function $\psi_{j,i}$ to be an Euclidean motion. In general, following~\cite{STV:2016}, transition functions are assumed to be piecewise bilinear mappings with respect to the mesh. To obtain a unique definition of edge lengths, we assume that all transition functions between structured charts are global Euclidean motions. Moreover, we assume that every element and every edge of the mesh is contained inside the domain $\domain{\Omega_i^P}$ of some structured chart $\omega_i$.

The concept of continuity of functions $\varphi:\Omega = \mathbf{G}(\Omega^P)\rightarrow\real$, defined over a parameter manifold $\Omega^P$, can now be introduced consistently. A function $\varphi$ is $C^k$-continuous across an edge $\edge \in \edges$, if $\pullback{\varphi}_i$ is $C^k$-continuous across ${\pi_i}^{-1}(\edge)$ for all structured charts $\omega_i$ that contain the edge (cf.~Definition~\ref{df: continuity edge}). This definition is consistent since all transition functions between structured charts are assumed to be $C^\infty$. This generalization allows one to define a spline space as in Definition~\ref{df: spline space T-mesh}, which is solely characterized by underlying structured charts, and extends the construction in Section~\ref{sec: Manifold splines}. For more rigorous definitions of structured/unstructured meshes and spline spaces on parameter manifolds, we refer to~\cite{STV:2016}.

\section{Refining totally unstructured meshes}
\label{sec: other meshes}

In this section, we investigate a variant of the proposed refinement scheme with rare violations of the condition \eqref{eq: di criterion} that is based on the manifold definition introduced in the previous section. 
This will enable us to use the refinement scheme on totally unstructured meshes (see Fig.~\ref{fig: tu mesh}) and when obeying a set of rules for these violations, we preserve the properties of the refinement algorithm presented above.

In order to identify violations of \eqref{eq: di criterion}, we generalize the definition of direction indices and replace Definition~\ref{df: direction index}, which works on the T-mesh $\mesh$ on the parameter manifold $\Omega$ by the following construction.
Recall Definition~\ref{df: proto-mesh} for the proto-mesh.
On each submesh $\tau_i$, $i=1,\dots,N$, the edges are assumed to be axis-parallel. By setting the direction index of an edge $\edgeSmall_i$ in $\tau_i$ to the actual direction of $\edgeSmall_i$, we find a direction labeling $\di_i$ that locally satifies \eqref{eq: di criterion}.
The condition \ref{pm: trans fncts map elems to elems} says that the transition functions map edges onto edges and each edge $\edge$ in the T-mesh $\mesh$ represents an equivalence class $\edge=\pi_i(\edgeSmall_i)$ of edges $\edgeSmall_i$ in the submeshes $\tau_i$. We get a set-valued direction labeling $\di:\edges\to2^\nat$, 
with $\di(\pi_i(\edgeSmall_i))=\{\di_j(\edgeSmall_j)\mid \edgeSmall_j\sim \edgeSmall_i\}$.

\begin{df}[admissible direction labeling]\label{df: adm dir labeling}%
	We call a direction labeling \emph{admissible} if 
	for any edge $\edge$, 
	the following conditions are fulfilled:
	\begin{itemize}
		\item Any adjacent edge $\edge'\in\edges(\nodes(\edge))$ of $\edge$ satisfies $\max\di(\edge')<\min\di(\edge)$ or $\min\di(\edge')>\max\di(\edge)$.
		\item In any neighbor element $\element\in\elements(\edge)$, the opposite edge $\edge'\in\edges(\element)\setminus\edges(\nodes(\edge))$ of $\edge$ satisfies $\max\di(\edge')\ge\min\di(\edge)$ and $\min\di(\edge')\le\max\di(\edge)$.
	\end{itemize}
\end{df}

If these conditions are met, Algorithm~\ref{alg: ref} can be applied, with ``$\di(\edge')<\di(\edge)$'' understood as $\max\di(\edge')<\min\di(\edge)$.
The results from section~\ref{sec: mesh properties} are still valid in this case, since the same interpretation is applicable in the proofs.
Note, however, that the relation ``$\max\di(\edge')\ge\min\di(\edge)$ and $\min\di(\edge')\le\max\di(\edge)$'' instead of $\di(\edge')=\di(\edge)$ in the second bullet point above is not a transitive relation, but only applicable elementwise.
We refer to Fig.~\ref{fig: tumesh gdi} for an example with the totally unstructured mesh from Fig.~\ref{fig: tu mesh}.

\begin{figure}[t]
	\centering
	\includegraphics[width=.19\textwidth]{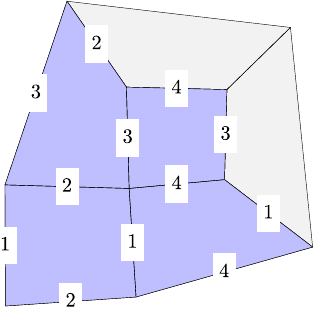}\hspace{.01\textwidth}%
	\includegraphics[width=.19\textwidth]{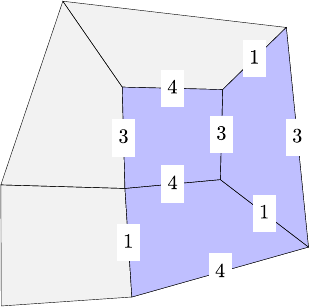}\hspace{.01\textwidth}%
	\includegraphics[width=.19\textwidth]{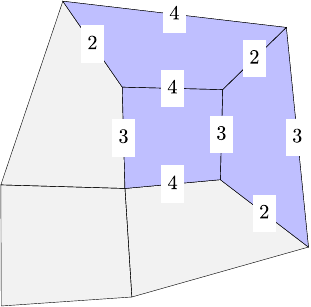}\hspace{.01\textwidth}%
	\includegraphics[width=.19\textwidth]{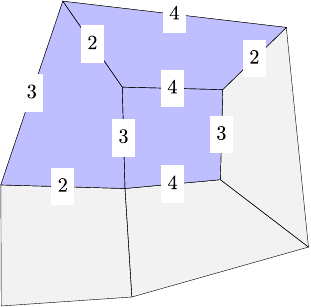}\hspace{.01\textwidth}%
	\includegraphics[width=.19\textwidth]{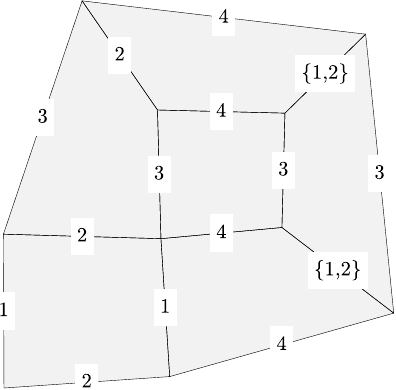}\\
	\includegraphics[height=.45\textwidth]{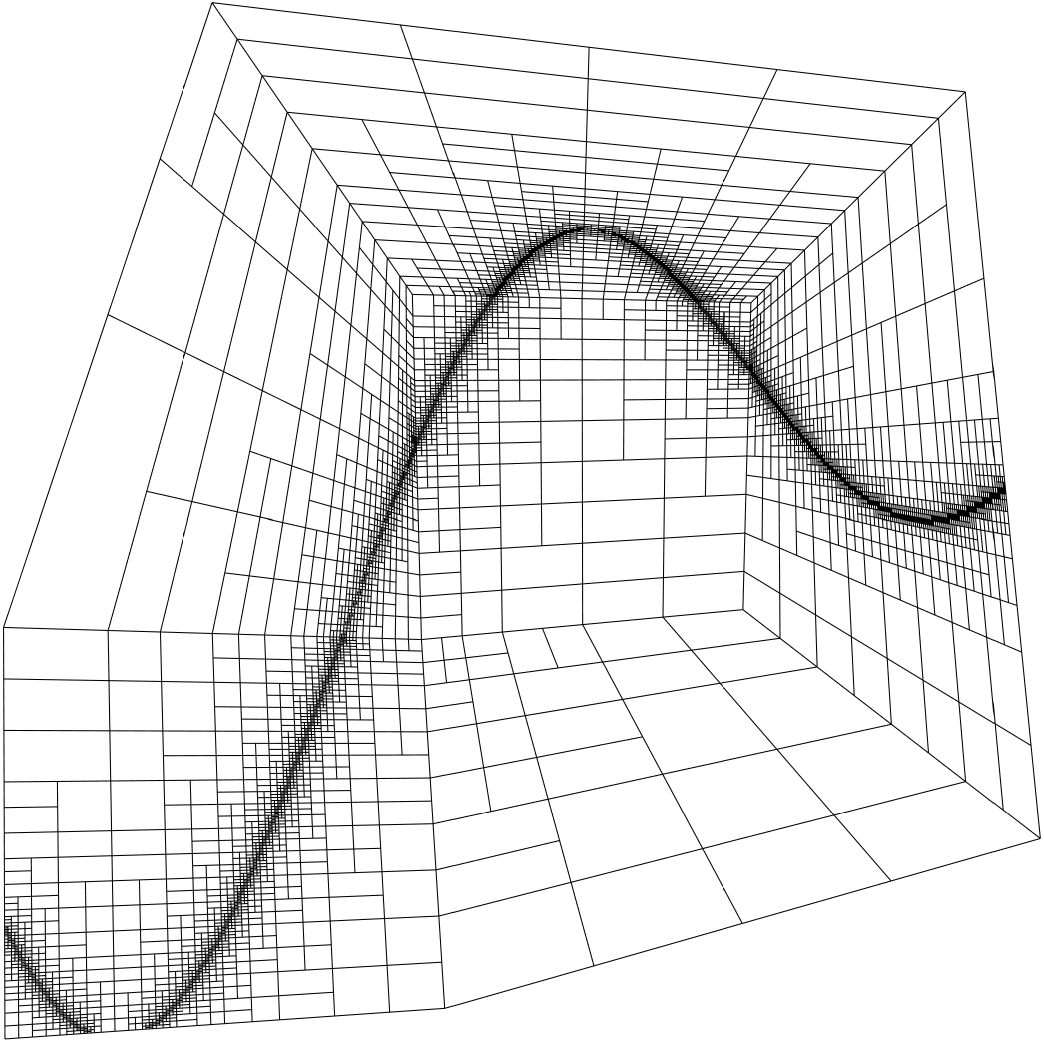}\hspace{.05\textwidth}%
	\includegraphics[height=.45\textwidth]{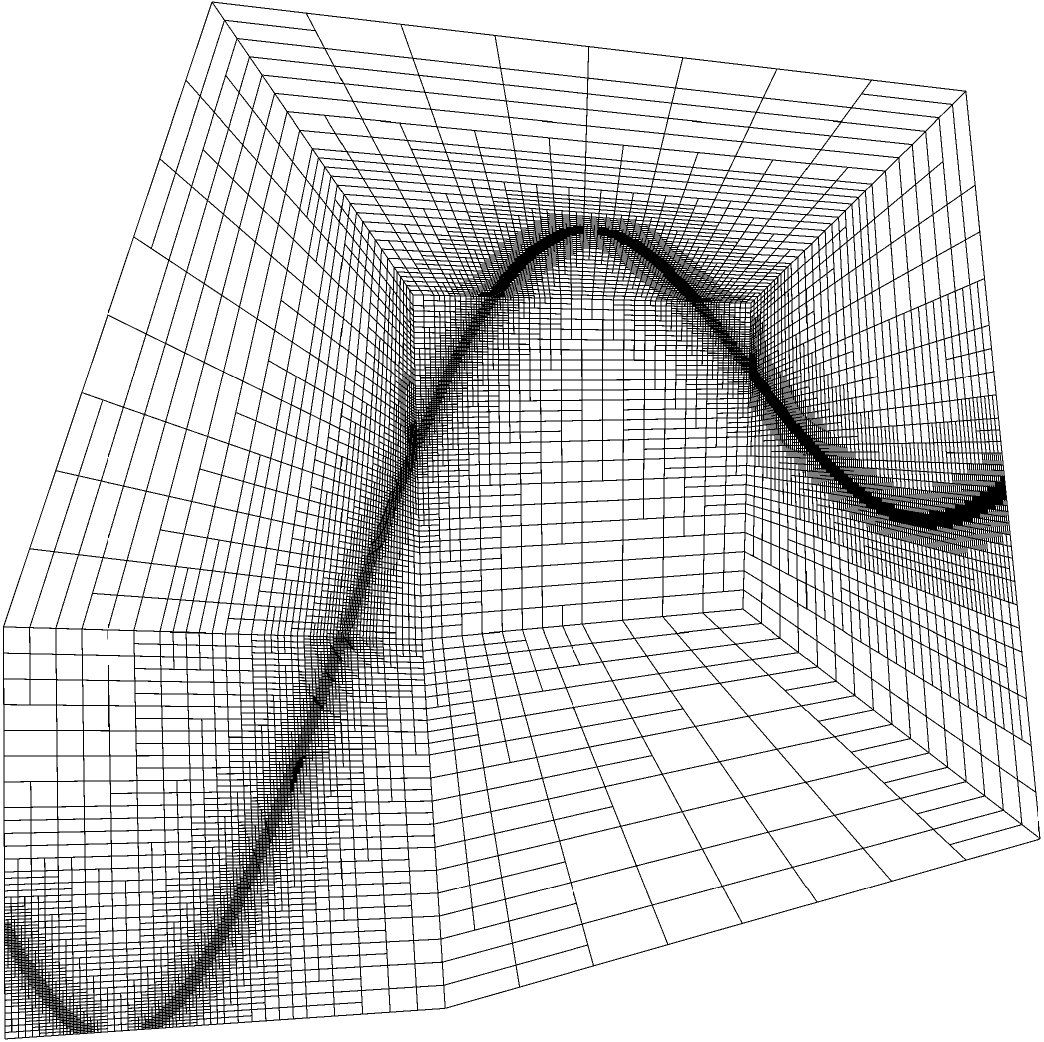}
	\caption{Although the mesh is totally unstructured, we can find a direction labeling that satisfies \eqref{eq: di criterion} locally and the associated generalized direction labeling satisfies Definition~\ref{df: adm dir labeling} (top row). This can be applied for refinement, here for $p=1$ (bottom left) and $p=3$ (bottom right).}
	\label{fig: tumesh gdi}
\end{figure}
For the deduction of direction indices, we further elaborate on the approach from Definition~\ref{df: direction index}. Regarding the mesh as a plane graph $G_\mesh$, the above-mentioned approach can be explained via the dual graph $G^\star$, which contains a node for each element in $G_\mesh$, including a node for the exterior element which neighbors all boundary edges. Two nodes in $G^\star$ are connected if the corresponding elements in $G_\mesh$ have a common edge. Hence there is a one-to-one correspondence between the edges of $G_\mesh$ and the edges of $G^\star$. The condensation of edges into equivalence classes $[E]$ corresponds to a circuit decomposition of $G^\star$, namely the unique decomposition into smallest circuits that traverse all nodes in a way, such that subsequent edges have no common element. Circuits are defined as follows. 
\begin{itemize}
	\item A walk is a sequence $(v_1,e_1,\dots,e_n,v_{n+1})$ such that $v_i$ and $v_{i+1}$ are boundary nodes of the edge $e_i$ for all $i=1,\dots,n$.
	\item A trail is a walk without edge repetitions.
	\item A path is a trail without node repetitions.
	\item A circuit is a closed trail, i.e. a trail $(v_1,e_1,\dots,e_n,v_{n+1})$ with $v_1=v_{n+1}$.
\end{itemize}
In order to eliminate self-crossings of the circuits in the totally unstructured case, we subdivide the circuits into overlapping paths that are not self crossing, see Figure~\ref{fig: finding di hard}. Afterwards, the corresponding touch graph is colored with the additional constraint that for any edge with multi-valued direction with minimal value $a$ and maximal value $b$ must not have a direction index in the range $[a,b]$, all adjacent edges must not have direction indices in the range $[a,b]$.

\begin{figure}
	\centering
	\includegraphics[width=.26\textwidth]{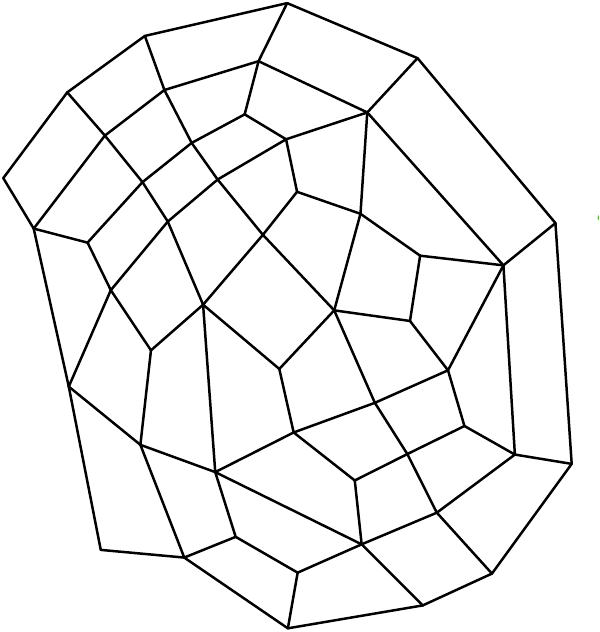}%
	\hspace{.07\textwidth}%
	\includegraphics[width=.26\textwidth]{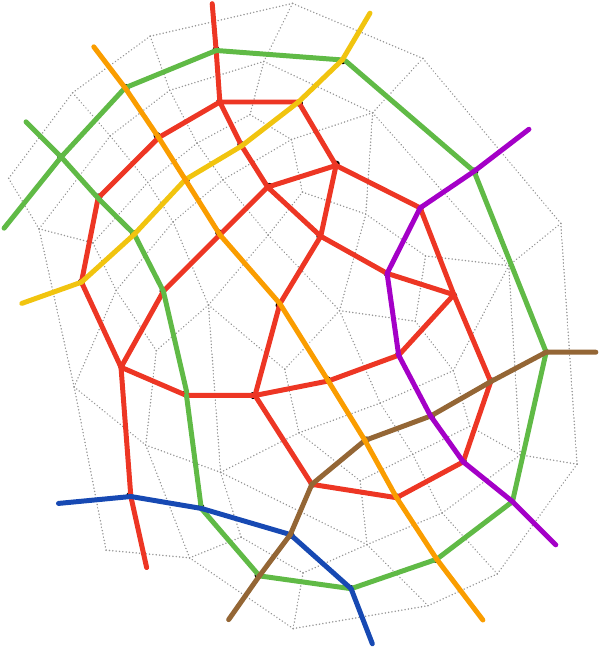}%
	\hspace{.07\textwidth}%
	\includegraphics[width=.26\textwidth]{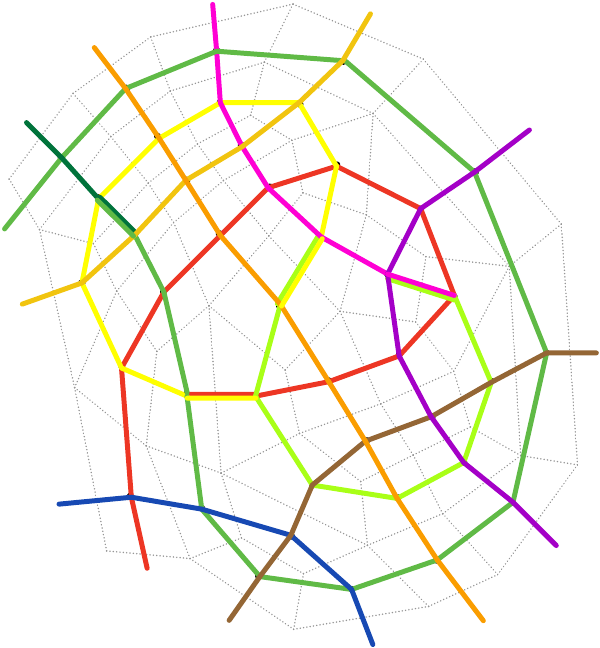}%
	\hspace{.07\textwidth}%
	\includegraphics[width=.26\textwidth]{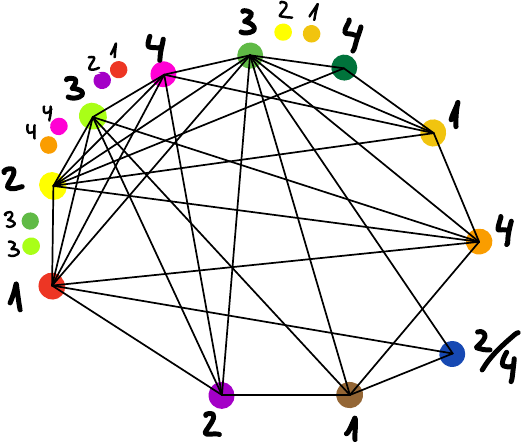}%
	\hspace{.07\textwidth}%
	\includegraphics[width=.26\textwidth]{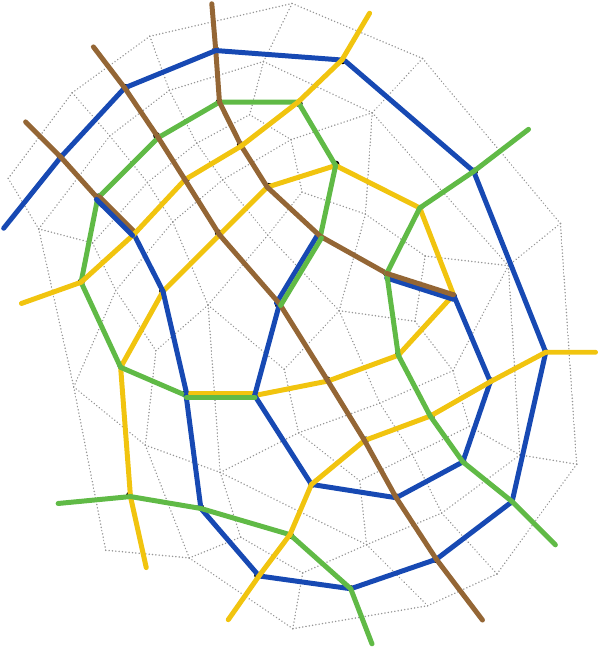}%
	\hspace{.07\textwidth}%
	\includegraphics[width=.26\textwidth]{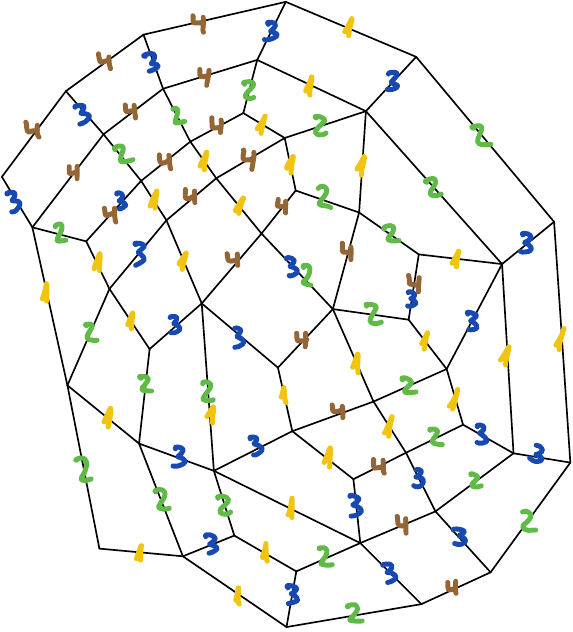}%
	\caption{For the deduction of direction indices in totally unstructured meshes, we subdivide the circuits of the dual graph into overlapping paths. Then, we solve a constrained coloring problem on the touch graph of the paths.
		The touch graph contains a node for each path, and two nodes are connected if the corresponding paths have a common node except for the exterior node. The constrained coloring problem seeks a labeling $c:\text{paths}\to\nat$ such that any two paths $p,q$ with a common non-exterior node satisfy $c(p)\ne c(q)$, and that for any edge $e$ that connects vertices $v,w$ and that is in two distinct, overlapping paths $p,q$, any other path $r$ containing one of the nodes $v,w$ satisfies $c(r)<\min(c(p),c(q))$ or $c(r)>\max(c(p),c(q))$. This constraint corresponds to the second bullet point in Definition~\ref{df: adm dir labeling}.
	}
	\label{fig: finding di hard}
\end{figure}

\section{Geometry mapping and isogeometric functions}\label{sec: geometry}

As T-splines are used in isogeometric analysis, one also has to take into account a geometry mapping which maps T-splines from the (manifold) domain $\Omega$ onto a physical domain of interest $D$.

We consider a T-mesh $\mesh$ and a corresponding T-spline basis $\cB$. We define a mapping
\begin{equation*}
	\mathbf{G} : \Omega \rightarrow D \subset \real^n,
\end{equation*}
where $n$ is the spatial dimension of the physical domain of interest such that 
\begin{equation*}
	\mathbf{G} \in \left(\mathrm{span}(\cB)\right)^n.
\end{equation*}
We can now define mapped elements of the T-mesh and mapped (isogeometric) T-splines over $D$ as
\begin{equation*}
	\mathcal{T} = \{ \mathbf{G}(\element) \,:\, \element \in \elements \}
\end{equation*}
and 
\begin{equation*}
	\cB(\mathcal{T}) = \{ B \circ \mathbf{G}^{-1} \,:\, B \in \cB \},
\end{equation*}
respectively. 

Let $\mesh'$ be a refinement of $\mesh$ and let $\cB(\mesh')$ be the refined T-spline basis. For many configurations, e.g., near extraordinary nodes, it may not be feasible to assume nestedness of the spaces, i.e., 
\begin{equation}
	\mathrm{span}(\cB(\mesh)) \not\subseteq \mathrm{span}(\cB(\mesh')).
\end{equation}
Consequently, the mapping $\mathbf{G'}\colon \Omega \rightarrow D'$ may not be equal to $\mathbf{G}$ but only an approximation. If the dimensions of $\Omega$ and $D$ are equal ($d = n$), it is reasonable to assume that the physical domain remains unchanged, i.e., $D = D'$. This is usually not possible if the domain is a surface in space, for instance. In both cases, the non-nestedness of the underlying spline spaces has to be taken into account when refining, cf.~\cite{YT:2017}.

\end{document}